\documentclass[11pt,reqno]{amsart}
\usepackage{latexsym,amssymb,graphicx,multicol,mathrsfs,color,xypic,amscd, amsxtra, verbatim}

\theoremstyle{plain}

\newtheorem{theorem}{Theorem}[section]
\newtheorem{lemma}[theorem]{Lemma}
\newtheorem{corollary}[theorem]{Corollary}
\newtheorem{proposition}[theorem]{Proposition}

\newtheorem*{problem}{Problem}

\theoremstyle{definition}
\newtheorem{remark}[theorem]{Remark}
\newtheorem{example}[theorem]{Example}
\newtheorem{definition}[theorem]{Definition}

\theoremstyle{remark}

\input xypic
\xyoption{all}

\def\Aut{\text{Aut\,}}
\def\Aff{\text{Aff\,}}

\def\Sec{\text{Sec\,}}
\def\Tan{\text{Tan\,}}
\def\dim{\text{dim\,}}
\def\D{\mathcal{D}}
\def\Def{\text{ \rm Def\,}}

\def\C{\mathcal{C}}

\def\Def{\text{Def\,}}

\def\Exc{\text{Exc\,}}
\def\Ext{\text{Ext\,}}
\def\Exal{\text{Exal\,}}

\def\Hom{\mathscr{H}\!om}
\def\Hilb{\mathscr{H}\!ilb}
\def\hom{\text{\rm Hom}}
\def\Isom{\mathcal{I}\!som}

\def\Image{\text{\rm Image}}

\def\pn{\{p_i\}_{i=1}^{n}}
\def\spn{\{p_i^s\}_{i=1}^{n}}

\def\L{\mathscr{L}}
\def\M{\mathcal{M}}
\def\N{\text{N}}

\def\SM{\overline{\mathcal{M}}}
\def\SMzero{\overline{M}_{0,n}}

\def\Nef{\overline{\N}^{1}_{+}(\C/\SM_{g,n})}
\def\EffCurves{\overline{\N}_{1}^{+}(\C/\SM_{g,n})}

\def\O{\mathscr{O}}

\def\P{\mathbb{P}}
\def\Q{\mathbb{Q}}
\def\X{\mathcal{X}}
\def\A{\mathcal{A}}
\def\SV{\mathcal{V}}
\def\U{\mathcal{U}}
\def\V{\mathcal{V}}

\def\Z{\mathcal{Z}}

\def\Aff{\text{Aff}}
\def\sigman{\{\sigma_{i}\}_{i=1}^{n}}
\def\sigmans{\{\sigma^s_{i}\}_{i=1}^{n}}
\def\sigmanprime{\{\sigma'_{i}\}_{i=1}^{n}}
\def\taun{\{\tau_{i}\}_{i=1}^{n}}

\def\Spec{\text{\rm Spec\,}}

\def\Pic{\text{\rm Pic\,}}

\begin{document}
\title{Towards a Classification of Modular Compactificatons of $\mathcal{M}_{g,n}$}
\author{David Ishii Smyth}

\maketitle
\begin{abstract} The moduli space of smooth curves admits a beautiful compactification $\M_{g,n} \subset \SM_{g,n}$ by the moduli space of stable curves. In this paper, we undertake a systematic classification of alternate modular compactifications of $\M_{g,n}$. Let $\U_{g,n}$ be the (non-separated) moduli stack of all $n$-pointed reduced, connected, complete, one-dimensional schemes of arithmetic genus $g$. When $g=0$, $\U_{0,n}$ is irreducible and we classify all open proper substacks of $\U_{0,n}$. When $g \geq 1$, $\U_{g,n}$ may not be irreducible, but there is a unique irreducible component $\V_{g,n} \subset \U_{g,n}$ containing $\M_{g,n}$. We classify open proper substacks of $\V_{g,n}$ satisfying a certain stability condition.
\end{abstract}
\tableofcontents

\pagebreak 
\section{Introduction}
\emph{Notation: An $n$-pointed curve $(C, \pn)$ is a reduced, connected, complete, one-dimensional scheme of finite-type over an algebraically closed field, together with a collection of $n$ points $p_1, \ldots, p_n \in C$. The marked points need not be smooth nor distinct. We say that a point on $C$ is \textit{distinguished} if it is marked or singular, and that a point on the normalization $\tilde{C}$ is \textit{distinguished} if it lies above a distinguished point of $C$. If $C$ is any curve and $Z \subset C$ is a proper subcurve, we call $Z^{c}:=\overline{C \backslash Z}$ \textit{the complement of $Z$}.}\\

\subsection{Statement of main result}\label{S:MainResult}

One of the most beautiful and influential theorems in modern algebraic geometry is the construction of a modular compactification $\M_{g,n} \subset \SM_{g,n}$ for the moduli space of smooth curves \cite{DeligneMumford}. The key point in this construction is the identification of a suitable class of singular curves, namely \emph{Deligne-Mumford stable curves}, with the property that every incomplete one-parameter family of smooth curves has a unique limit contained in this class. While the class of stable curves gives a natural modular compactification of the space of smooth curves, it is not unique in this respect. There exist two alternate compactifications in the literature, the moduli space of pseudostable curves \cite{Schubert}, in which cusps arise, and the moduli space of weighted pointed curves, in which sections with small weight are allowed to collide \cite{Hassett1}. In light of these constructions, it is natural to ask

\begin{problem}
Can we classify all possible stability conditions for curves, i.e. classes of singular marked curves which are deformation-open and satisfy the property that any one-parameter family of smooth curves contains a unique limit contained in that class.
\end{problem} 

Stable, pseudostable, and weighted stable curves all have the property that every rational component of the normalization has at least three distinguished points. In general, we say that an $n$-pointed curve with this property is \emph{prestable}. The main result of this paper classifies stability conditions on prestable curves, i.e. we give a simple combinatorial description of all deformation-open classes of prestable curves with the unique limit property.

Stability conditions on curves correspond to open proper substacks of the moduli stack of all curves. To make this precise, let $\U_{g,n}$ be the functor from schemes to groupoids defined by
\begin{equation*}
\U_{g,n}(T):=
\left\{
\begin{matrix}
\text{ Flat, proper, finitely-presented morphisms $\C \rightarrow T$, with $n$ sections}\\
\text{$\{ \sigma_i\}_{i=1}^{n}$, and connected, reduced, one-dimensional geometric fibers.}\\
\end{matrix}
\right\}
\end{equation*}
Note that we always allow the total space $\C$ of a family to be an algebraic space. In Appendix B, it is shown that $\U_{g,n}$ is an algebraic stack, locally of finite-type over $\Spec \mathbb{Z}$. Let $\M_{g,n} \subset \U_{g,n}$ denote the open substack corresponding to families of smooth curves. Since $\M_{g,n}$ is irreducible, there is a unique irreducible component $\V_{g,n} \subset \U_{g,n}$ containing $\M_{g,n}$. The points of $\V_{g,n}$ correspond to smoothable curves, while the generic point of every extraneous component of $\U_{g,n}$ parametrizes a non-smoothable curve. Since we are interested in irreducible compactifications of $\M_{g,n}$, we work exclusively in $\V_{g,n}$.

\begin{definition} A \emph{modular compactification of $\M_{g,n}$} is an open substack $\X \subset \V_{g,n}$, such that $\X$ is proper over $\Spec \mathbb{Z}.$ 
 \end{definition}
 
Since the definition of a modular compactification is topological, the set of modular compactifications of $\M_{g,n}$ does not depend on the stack structure of $\V_{g,n}$. In the absence of a functorial description for $\V_{g,n}$, we simply define the stack structure by taking the stack-theoretic image in $\U_{g,n}$ of the inclusion $\V_{g,n}^{0} \hookrightarrow \U_{g,n}$, where $\V_{g,n}^{0}$ is the interior of $\V_{g,n}$, i.e.
$
\V_{g,n}^{0}:=\V_{g,n} \backslash (\V_{g,n} \cap \overline{\U_{g,n} \backslash \V_{g,n}}).
$
 
The long-term goal of this project is to classify all modular compactifications of $\M_{g,n}$. This paper takes the first step by classifying all \emph{stable modular compactifications of $\M_{g,n}$}. If $(C,\{p_i\}_{i=1}^{n})$ is an $n$-pointed curve, we say that $(C,\{p_i\}_{i=1}^{n})$ is \emph{prestable} (resp. \emph{presemistable}) if every rational component of the normalization $\tilde{C}$ has at least three (resp. two) distinguished points. We then define the restricted class of stable (resp. semistable) modular compactifications as follows.

\begin{definition}
A modular compactification $\X \subset \V_{g,n}$ is \emph{stable} (resp. \emph{semistable}) if every geometric point $[C,\{p_i\}_{i=1}^{n}] \in \X$ is prestable (resp. presemistable).
\end{definition}

\begin{remark}
It is by no means obvious that there should exist strictly semistable modular compactifications of $\M_{g,n}$. After all, if a nodal curve $(C, \pn)$ contains a smooth rational subcurve with only two distinguished points, then $\Aut(C, \pn)$ is not proper; in particular, $(C,\pn)$ cannot be contained in any proper substack of $\U_{g,n}$. A similar argument (see Section \ref{SS:GenusZero}) shows that every modular compactification of $M_{0,n}$ is stable, so the methods of this paper give a classification of \emph{all} modular compactifications of $M_{0,n}$. By contrast, the author has constructed a sequence of strictly semistable modular compactifications of $\M_{1,n}$ \cite{Smyth1}. Thus, for $g \geq 1$, our classification of stable compactifications does not tell the whole story.
\end{remark}

We should observe that a stable modular compactification $\X$ necessarily has quasi-finite diagonal. In particular, $\X$ admits an irreducible coarse moduli space $X$, which gives a proper (though not necessarily projective) birational model of $M_{g,n}$ \cite{KeelMori}.
\begin{proposition} If $\X$ is a stable modular compactification, then $\X$ has quasi-finite diagonal.

\end{proposition}
\begin{proof}
We must show that if $(C,\pn)$ is a prestable curve over an algebraically closed field $k$, the group scheme $\Aut_{k}(C,\pn)$ has finitely many $k$-points. Let $\{\tilde{q}_i\}_{i=1}^{m} \in \tilde{C}$ be the set of distinguished points of $\tilde{C}$. Any $k$-automorphism of $(C,\pn)$ induces a $k$-automorphism of $\tilde{C}$ mapping the set $\{\tilde{q}_i\}_{i=1}^{m}$ to itself. Since each rational component of $\tilde{C}$ has at least three distinguished points, and each elliptic component has at least one distinguished point, the set of such automorphisms is finite.
\end{proof}

\begin{comment}
We can give an equivalent definition of stable and semistable modular compactifications by considering the functor from schemes to groupoids
\begin{equation*}
\U_{g,n}^{red}(T):=
\left\{
\begin{matrix}
\text{Flat proper finitely-presented morphisms $X \rightarrow T$, with $n$ sections}\\
\text{$\{ \sigma_i\}_{i=1}^{n}$, and geometrically connected and reduced 1-dimensional fibers.}\\
\end{matrix}
\right\}
\end{equation*}
$\U^{red}_{g,n}$ which is the moduli stack of reduced connected curves and it is immediate to check that $\U^{red}_{g,n} \subset \U_{g,n}$ is an open substack. As before, we define $\V_{g,n}^{red} =\V_{g,n} \cap \U^{red}_{g,n}$ to be the unique irreducible component containing $\M_{g,n}.$ Then we have
\begin{lemma}
Stable (resp. semistable) modular compactifications $\X \subset \V_{g,n}$ correspond naturally to open substacks $\X \subset \V_{g,n}^{red}$ which satisfy
 \item $\M_{g,n} \subset \X$
 \item $\X$ is proper over $\Spec \mathbb{Z}.$ 
\end{lemma}
\begin{proof}
If $\X$ is a semistable modular compactification, then obviously $\X$ is open in $\V_{g,n}$ and $\U_{g,n}^{red}$, hence also in $\V_{g,n}^{red}$. It only remains to check that if $\X$ is proper over $\Spec \mathbb{Z}$, then $\X \supset \M_{g,n}$. This is proved in corollary ??
\end{proof}
\end{comment}

Now let us describe the combinatorial data that goes into the construction of a stable modular compactification. If $(C, \pn)$ is a Deligne-Mumford stable curve, we may associate to $(C,\pn)$ its \emph{dual graph} $G$. The vertices of $G$ correspond to the irreducible components of $C$, the edges correspond to the nodes of $C$, and each vertex is labeled by the arithmetic genus of the corresponding component as well as the marked points supported on that component. The dual graph encodes the topological type of $(C, \pn)$, and for any fixed $g,n$, there are only finitely many isomorphism classes of dual graphs of $n$-pointed stable curves of genus $g$. 

\begin{figure}
\scalebox{.40}{\includegraphics{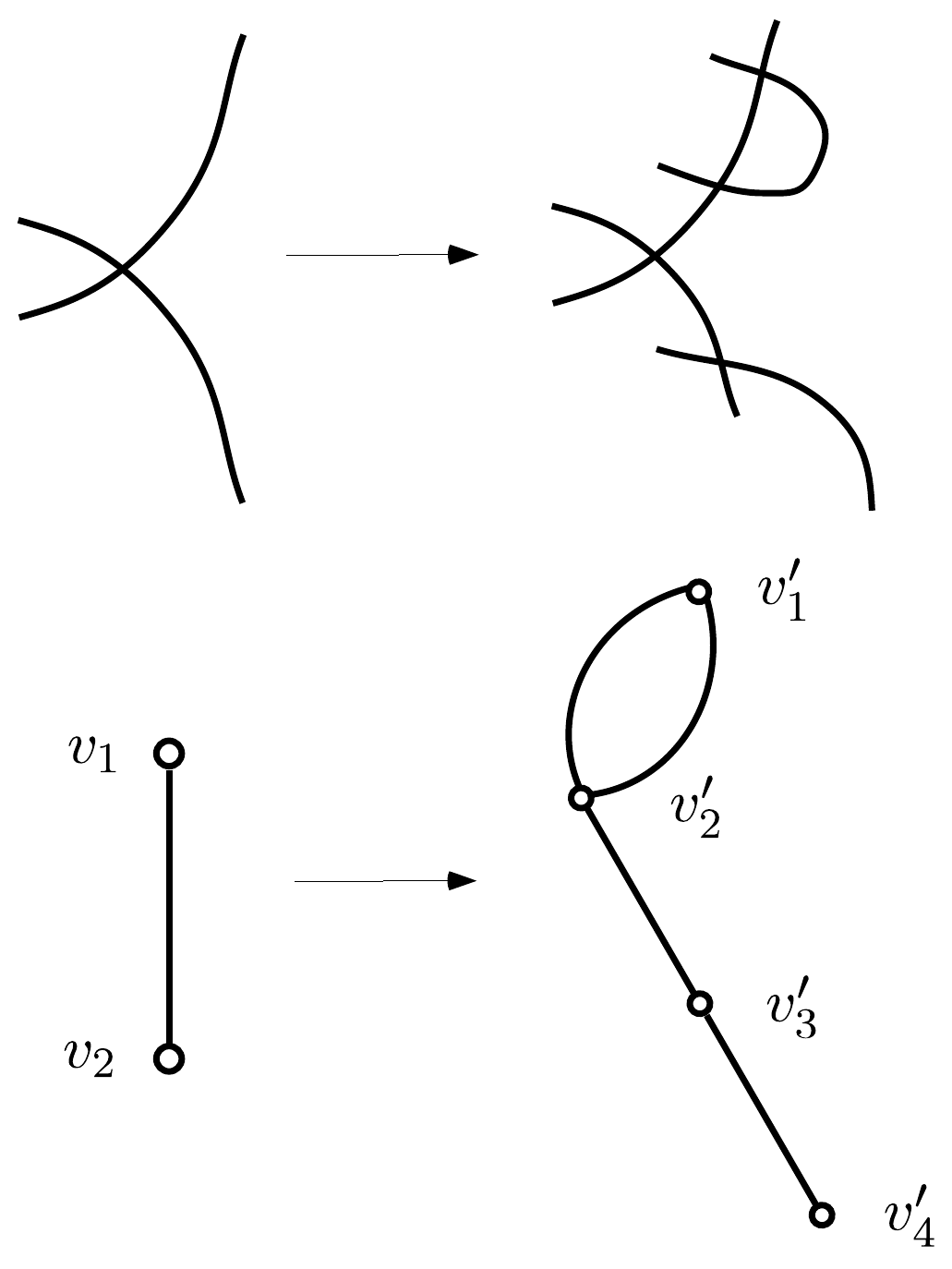}}
\caption{The specialization of dual graphs induced by a one-parameter specialization of stable curves. Note that $v_{1} \leadsto v_{1}' \cup v_{2}'$ and $v_{2} \leadsto v_{3}' \cup v_{4}'$.}\label{F:GraphSpecialization}
\end{figure}

We write $G \leadsto G'$ if there exists a stable curve $(\C \rightarrow \Delta, \sigman)$ over the spectrum of a discrete valuation ring with algebraically closed residue field, such that the geometric generic fiber has dual graph $G$ and the special fiber has dual graph $G'$. If $v$ is a vertex of $G$, and we have $G \leadsto G'$, we say that $G \leadsto G'$ induces $v \leadsto v_1' \cup \ldots \cup v_k'$ to indicate that the limit of the irreducible component corresponding to $v$ is the union of the irreducible components corresponding to $v_1', \ldots, v_k'$ (See Figure \ref{F:GraphSpecialization}). More precisely, if $\C \rightarrow \Delta$ is a one-parameter family witnessing the specialization $G \leadsto G'$, then (possibly after a finite base-change) we may identify the irreducible components of the geometric generic fiber with the irreducible components of $\C$. In particular, $v$ corresponds to an irreducible component $\C_{1} \subset \C$, and the limit of $v$ is simply the collection of irreducible components in the special fiber of $\C_{1}$.  Now we come to the key definition of this paper.
\begin{definition}[Extremal assignment over $\SM_{g,n}$]\label{D:Assignment}
Let $G_1, \ldots, G_N$ be an enumeration of dual graphs of $n$-pointed stable curves of genus $g$, up to isomorphism, and consider an assignment 
$$G_{i} \rightarrow \Z(G_i) \subset G_i, \text{ for each $i=1, \ldots, N$,}$$
where $\Z(G_i)$ is a subset of the vertices of $G_i$. We say that $\Z$ is an \emph{extremal assignment over $ \SM_{g,n}$} if it satisfies the following three axioms.
\begin{itemize}
\item[(1)] For any dual graph $G$, $\Z(G) \neq G$.
\item[(2)] For any dual graph $G$, $\Z(G)$ is invariant under $\Aut(G)$.
\item[(3)] For every specialization $G \leadsto G'$, inducing $v \leadsto v_1' \cup \ldots \cup v_k'$, we have $$v \in \Z(G) \iff v_1', \ldots, v_k' \in \Z(G').$$ 

\end{itemize}
\end{definition}
\begin{remark}
Axiom 2 in Definition \ref{D:Assignment} implies that an extremal assignment $\Z$ determines, for each stable curve $(C,\pn)$, a certain subcurve $\Z(C) \subset C.$ Indeed, we may chose an isomorphism of the dual graph of $(C,\pn)$ with $G_{i}$ for some $i$, and then define $\Z(C)$ to be the collection of irreducible components of $C$ corresponding to $\Z(G_i)$ under this isomorphism. By axiom 2, $\Z(C)$ does not depend on the choice of isomorphism.
\end{remark}

Given an extremal assignment $\Z$, we say that a curve is \emph{$\Z$-stable} if it can be obtained from a Deligne-Mumford stable curve $(C, \{p_i\}_{i=1}^{n})$ by replacing each connected component of $\Z(C) \subset C$ by an isolated curve singularity whose contribution to the arithmetic genus is the same as the subcurve it replaces. We make this precise in Definitions \ref{D:Genus} and \ref{D:Zstable} below.
 
\begin{definition}[Genus of a curve singularity]\label{D:Genus}
Let $p \in C$ be a point on a curve, and let $\pi: \tilde{C} \rightarrow C$ denote the normalization of $C$ at $p$. The $\delta$-invariant $\delta(p)$ and the number of branches $m(p)$ are defined by following following formulae:
\begin{align*}
\delta(p)&:=\dim_{k} (\pi_*\O_{\tilde{C}}/\O_{C}),\\
m(p)&:=|\pi^{-1}(p)|,
\end{align*}
and we define the \emph{genus} $g(p)$ by
\begin{align*}
g(p)&:=\delta(p)-m(p)+1.
\end{align*}
We say that a singularity $p \in C$ has \emph{type} $(g,m)$ if $g(p)=g$ and $m(p)=m$.
\end{definition}

\begin{definition}[$\Z$-stability]\label{D:Zstable}
A smoothable $n$-pointed curve $(C, \pn)$ is \emph{$\Z$-stable} if there exists a stable curve $(C^{s}, \spn)$ and a morphism  $\phi: (C^{s}, \spn) \rightarrow (C, \pn)$ satisfying
\begin{enumerate}
\item $\phi$ is surjective with connected fibers.
\item $\phi$ maps $C^{s}-\Z(C^{s})$ isomorphically onto its image.
\item If $Z_{1}, \ldots, Z_{k}$ are the connected components of $\Z(C^{s})$, then $p_i:=\phi(Z_i) \in C$ satisfies $g(p_i)=p_a(Z_i)$ and $m(p_i)=|Z_i \cap Z_i^c|$.
\end{enumerate}
\end{definition}

For any extremal assignment $\Z$, we define $\SM_{g,n}(\Z) \subset \V_{g,n}$ to be the set of points corresponding to $\Z$-stable curves. The following theorem is our main result.

\begin{theorem}[Classification of Stable Modular Compactifications]\label{T:Main}
\begin{itemize}
\item[]
\item[(1)] If $\Z$ is an extremal assignment over $\SM_{g,n}$, then $\SM_{g,n}(\Z) \subset \V_{g,n}$ is a stable modular compactification of $\M_{g,n}$.
\item[(2)] If $\X \subset \V_{g,n}$ is a stable modular compactification, then $\X=\SM_{g,n}(\Z)$ for some extremal assignment $\Z$.
\end{itemize}
\end{theorem}
\begin{proof}
See Theorems \ref{T:Construction} and \ref{T:Classification}.
\end{proof}
Since the definition of an extremal assignment is purely combinatorial, one can (in principal) write down all extremal assignments over $\SM_{g,n}$ for any fixed $g$ and $n$. Thus, we obtain a complete classification of the collection of stable modular compactifications of $\M_{g,n}$. Before proceeding, let us consider some examples of extremal assignments, and describe the corresponding stability conditions.

\begin{example}[Destabilizing elliptic tails]\label{E:FirstAssignments}
Consider the assignment defined by
$$
\Z(C, \pn)=\{ Z \subset C \,|\, p_a(Z)=1, |Z \cap Z^{c}|=1, Z \text{ is unmarked }\}.
$$
If we call an subcurve $Z \subset C$ satisfying $p_a(Z)=1$ and $|Z \cap Z^{c}|=1$ an \emph{elliptic tail}, we may say that the assignment $\Z$ is defined by picking out all unmarked elliptic tails of $(C, \pn)$. This defines an extremal assignment over $\SM_{g,n}$ provided that $g>2$ or $n>1$. The case $(g,n)=(2,0)$ is forbidden because an unmarked genus two curve may be the union of two elliptic tails. Assuming $(g,n) \neq (2,0)$, axioms 1 and 2 are obvious. For axiom 3, simply note that if $v \in G$ corresponds to an elliptic tail, then any specialization of dual graphs $G \leadsto G'$ necessarily induces a specialization $v \leadsto v'$, where $v'$ is also an elliptic tail. Thus, $v \in \Z(G) \iff v' \in \Z(G')$ as required.

Now let us consider the associated $\Z$-stability condition. By definition, an $n$-pointed curve $(C, \pn)$ is $\Z$-stable if there exists a map from a stable curve $(C^s, \spn) \rightarrow (C, \pn)$ which is an isomorphism away from the locus of elliptic tails, and contracts each elliptic tail of $C^{s}$ to a singularity of type $(1,1)$. It is elementary to check that the unique curve singularity of type $(1,1)$ is a cusp $(y^2-x^3)$. Thus, an $n$-pointed curve $(C, \pn)$ is \emph{$\Z$-stable} for this assignment iff it satisfies:
\begin{itemize}
\item[(1)]$C$ has only nodes and cusps as singularities.
\item[(2)] The marked points $\pn$ are smooth and distinct.
\item[(3)] Each rational component of $\tilde{C}$ has at least three distinguished points.
\item[(4)] If $E \subset C$ is an unmarked arithmetic genus one subcurve, $|E \cap E^{c}| \geq 2$.
\end{itemize}
When $n=0$, this is precisely the definition of \emph{pseudostability} introduced in \cite{Schubert} and further studied in \cite{Hassett3}.
\begin{comment}
In these papers, it is shown that
$$\SM_{g}(\Z) = \SM_{g}^{ps}= [\textbf{Hilb}_{4}^{s}\text{//}\text{PGL}(7g-7)],$$ where $\textbf{Hilb}^{4}_{s}$ is the stable locus in the Hilbert scheme of 4-canonically embedded curves. In addition, $\SM_{g}^{ps}$ is isomorphic to the log-canonical model of $\SM_{g}$ associated to the divisor $K_{\SM_{g}}+\alpha\Delta$ for $\alpha \in \Q \cap (7/10,9/11]$.
\end{comment}
\end{example}
\begin{example} (Destabilizing rational tails) Consider the assignment defined by
$$
\Z(C)=\{ Z \subset C \,|\, p_a(Z)=0, |Z \cap Z^{c}|=1, |\{p_i \in Z\}| \leq k \}.
$$
If we call a subcurve $Z \subset C$ satisfying $p_a(Z)=0$ and $|Z \cap Z^{c}|=1$ a \emph{rational tail}, we may say that the assignment $\Z$ is defined by picking out all rational tails of $(C, \pn)$ with $\leq k$ marked points. This defines an extremal assignment over $\SM_{g,n}$ provided that $g>0$ or $n>2k$. The case $g=0$ and $n \leq 2k$ is forbidden because such stable curves may be the union of two rational tails with $\leq k$ marked points. If $g>0$ or $n>2k$, axioms 1 and 2 are easily verified. Axiom 3 is also obvious, bearing in mind that we do not require a rational tail $Z \subset C$ to be irreducible.

Now let us consider the associated $\Z$-stability condition. An $n$-pointed curve $(C, \pn)$ is $\Z$-stable if there exists a map from a stable curve $(C^s, \spn) \rightarrow (C, \pn)$ which is an isomorphism away from the locus of rational tails with $\leq k$ points, and contracts each such rational tail to a point of type (0,1) on $C$. It follows directly from the definition that the unique `singularity' of type (0,1) is a smooth point. Thus, an $n$-pointed curve $(C, \pn)$ is $\mathcal{\Z}$-stable for this assignment iff it satisfies:
\begin{itemize}
\item[(1)]$C$ has only nodes as singularities.
\item[(2)] The marked points $\pn$ are smooth, and up to $k$ points may coincide.
\item[(3)] Each rational component of $\tilde{C}$ has at least three distinguished points.
\item[(4)] If $Z \subset C$ is a rational tail, then $|\{p_i: p_i \in Z\}| >k$.
\end{itemize}
This is equivalent to the definition of $\A$-stability introduced in \cite{Hassett1} with symmetric weights $\A=\{1/k, \ldots, 1/k\}$.
\end{example}

\begin{example}\label{E:CrazyAssignment} (Destabilizing all unmarked components) Consider the assignment defined by
$$
\Z(C, \pn)=\{ Z \subset C \,|\, Z\text{ is unmarked\! }\}.
$$
As long as $n\geq1$, this assignment clearly satisfies axioms 1-3 of Definition \ref{D:Assignment}. The corresponding $\Z$-stable curves have all manner of exotic singularities. In fact, for any pair of integers $(h,m)$, there exists $g>>0$ such that $n$-pointed stable curves of genus $g$ contain unmarked subcurves $Z \subset C$ satisfying $p_{a}(Z)=h$ and $|Z \cap Z^{c}|=m$. It follows that every smoothable curve singularity of type $(h,m)$ appears on a $\Z$-stable curve for $g>>0.$ The corresponding moduli spaces $\SM_{g,n}(\Z)$ have no counterpart in the existing literature.\\
\end{example}

\subsection{Consequences of main result}\label{S:Consequences}
In this section, we describe several significant consequences of Theorem \ref{T:Main}. First, we will show that the number of extremal assignments over $\SM_{g,n}$ is a rapidly increasing function of both $g$ and $n$ by explaining how $\pi$-nef line-bundles on the universal curve $\pi:\C \rightarrow \SM_{g,n}$ induce extremal assignments. We deduce the existence of many new stability conditions which have never been described in the literature. Next, we explain why $\Z$-stability nevertheless fails to give an entirely satisfactory theory of stability conditions for curves. We will see, for example, that there is no $\Z$-stability condition picking out only curves with nodes $(y^2=x^2)$, cusps $(y^2=x^3)$, and tacnodes $(y^2=x^4)$, and indicate how a systematic study of semistable compactifications might remedy this deficiency. Finally, we will show that $\Z$-stability \emph{does} give a satisfactory theory of stability conditions in the case $g=0$. We will see that every modular compactification of $M_{0,n}$ must be stable, so our result actually gives a complete classification of modular compactifications of $M_{0,n}$.

\subsubsection{Extremal assignments from $\pi$-nef line-bundles}\label{SS:NefAssignments}
Let $\SM_{g,n}$ denote the moduli stack of stable curves over an algebraically closed field of characteristic zero, and let $\pi:\C \rightarrow \SM_{g,n}$ denote the universal curve. The following lemma shows that numerically-nontrivial $\pi$-nef line-bundles on $\C$ induce extremal assigments. (In this context, to say that $\L$ is $\pi$-nef and numerically-nontrivial simply means that $\L$ has non-negative degree on every irreducible component of every fiber of $\pi$ and positive degree on the generic fiber.)

\begin{lemma}\label{L:NefAssignments}
Let $\L$ be a $\pi$-nef, numerically non-trivial line-bundle on $\C$. Then $\L$ induces an extremal assignment by setting:
$$
\Z(C,\pn):=\{Z \subset C |\, \deg(\L|_{Z})=0 \,\},
$$
for each stable curve $[C,\pn] \in \SM_{g,n}$.
\end{lemma}
\begin{proof}
We must check that the assignment $\Z$ satisfies axioms 1-3 in Definition \ref{D:Assignment}. For axiom 1, observe that since $\L$ is $\pi$-nef and numerically non-trivial, $\L$ must have positive degree on some irreducible component of each geometric fiber of $\pi$. For axiom 2, recall  that $\Pic_{\Q}(\C/\SM_{g,n})$ is generated by line-bundles whose degree on any irreducible component of a fiber of $\pi$ depends only on the dual graph of the fiber \cite{AC}. For axiom 3, consider any specialization $G \leadsto G'$ induced by a one-parameter family of stable curves $(\C^s \rightarrow \Delta, \sigman)$. We have a Cartesian diagram
\[
\xymatrix{
\C^s \ar[d] \ar[r]^{f}&\C \ar[d]\\
\Delta \ar[r]& \SM_{g,n},
}
\]
and after a finite base-change, we may assume that the irreducible components of $\C^{s}$ are in bijective correspondence with the irreducible components of the geometric generic fiber, i.e. we have $\C^s \simeq \C_{1} \cup \ldots \cup \C_{k}$ with each $\C_{i} \rightarrow \Delta$ having smooth generic fiber. Now $f^*\L$ has degree zero on the generic fiber of $\C_{i} \rightarrow \Delta$ iff it has degree zero on every irreducible component of the special fiber. This is precisely the statement of axiom 3.
\end{proof}

Translated into the language of higher-dimensional geometry, this lemma says that every face of the relative cone of curves $\EffCurves$ gives rise to a stable modular compactification of $\M_{g,n}$. In Appendix A, we give an explicit definition of $\EffCurves$ as a closed polyhedral cone in $\Pic_{\Q}(\C/\SM_{g,n})$, and describe the stability conditions corresponding to each extremal face in the cases $(g,n)=(2,0), (3,0), (2,1)$. In general, since $\EffCurves$ is a full polyhedral cone in a vector space of dimension $\rho(\C /\SM_{g,n})$, it is clear that the number of extremal faces of $\EffCurves$ (and hence the number of extremal assignments over $\SM_{g,n}$) is a rapidly increasing function of both $g$ and $n$.

\subsubsection{Singularities arising in stable compactifications}\label{SS:StableGeometry}
While Lemma \ref{L:NefAssignments} shows that there exist many stability conditions for curves, it does not provide much insight into the following natural question: Given a deformation-open class of curve singularities, is there a stability condition which picks out curves with precisely this class of singularities? We have already seen that the anwer is yes if the class consists of nodes or nodes and cusps. In general, however, one cannot always expect an affirmative answer using only stability conditions on prestable curves. Indeed, Corollaries \ref{C:1} and \ref{C:2} below show that the collection of stable modular compactifications of $\M_{g,n}$ is severely constrained by two features: the necessity of compactifyng the moduli of attaching data of a singularity (a local obstruction) and the presence of symmetry in dual graphs of stable curves (a global obstruction).

Let us say that a given curve singularity \emph{arises in a modular compactification $\X$} if $\X$ contains a geometric point $[C, \pn] \in \X$ such that $C$ possesses this singularity.
\begin{corollary}\label{C:1}
Let $\X$ be a stable modular compactification of $\M_{g,n}$. If one singularity of type $(h,m)$ arises in $\X$, then every singularity of type $(h,m)$ arises in $\X$.
\end{corollary}
\begin{proof}
We have $\X=\SM_{g,n}(\Z)$ for some extremal assignment $\Z$. If a singularity of type $(h,m)$ appears on some $\Z$-stable curve, there exists a stable curve $(C^{s}, \spn)$ and a connected component $Z \subset \Z(C^{s})$ such that $p_a(Z)=h$ and $|Z \cap Z^{c}|=m$. Since the definition of $\Z$-stability allows $Z$ to be replaced by any singularity of type $(h,m)$, it follows that \emph{all} singularities of type $(h,m)$ arise in $\X$.
\begin{comment}
To see that all singularities of type $(h',m') \leq (h,m)$ arise in $\X$, it suffices to exhibit a stable curve $D^{s}$ with a connected component $Z' \subset \Z(D^{s})$ such that $p_a(Z')=h'$ and $|Z' \cap (Z')^{c}|=m'.$ We produce $D^{s}$ in two steps: first specialize $C^{s}$ so that $Z$ splits off a subcurve $Z'$ with $p_a(Z')=h'$ and $|Z' \cap (Z')^{c}|=m',$ then smooth the nodes external to $Z'$. Applying axiom (3) of Definition \ref{D:Assignment} to this pair of specializations, we conclude that $\Z(D^{s})=Z'$.
\end{comment}
\end{proof}
\begin{figure}
\scalebox{.50}{\includegraphics{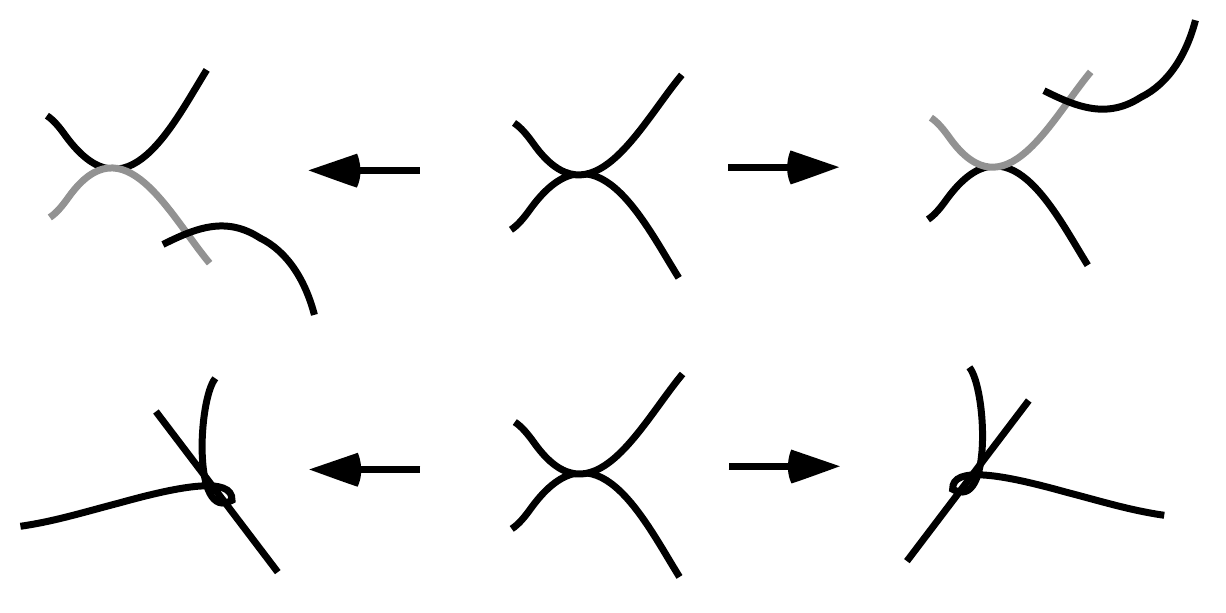}}
\caption{Two methods for compactifying the $k^*$-moduli of attaching data of the tacnode. In a stable modular compactification, one must compactify by degenerating to a cusp with a transverse branch. In  a semistable modular compactification, one may compactify by allowing the normalization to sprout additional rational components.}\label{F:AttachingData}
\end{figure}
This corollary precludes the existence of a stability condition on prestable curves picking out precisely nodes, cusps, and tacnodes. Indeed, one easily checks that the spatial singularity obtained by passing a smooth branch through the tangent plane of a cusp, i.e.
$$\hat{\O}_{C,p} \simeq k[[x,y,z]]/((x,y) \cap (z,y^2-x^3)),$$
has the same genus (1) and number of branches (2) as the tacnode. Thus, any stability condition on prestable curves which allows tacnodes must allow this spatial singularity as well.

The geometric phenomenon responsible for this implication is the existence of moduli of `attaching data' for a tacnode. Unlike nodes or cusps, the isomorphism class of a tacnodal curve $C$ is not uniquely determined by its pointed normalization $(\tilde{C},q_1,q_2)$; one must also specify an element $\lambda \in \text{Isom}(T_{q_1}\tilde{C}, T_{q_2}\tilde{C}) \simeq k^{*}$. As $\lambda \rightarrow 0$ or $\infty$, the tacnodal curve degenerates into a cusp with a transverse branch (see Figure \ref{F:AttachingData}). Note, however, that in a semistable compactification, one may compactify moduli of attaching data by sprouting additional rational components (see Figure \ref{F:AttachingData}). Indeed, this alternate method of compactification is used in \cite{Smyth1} to construct strictly semistable modular compactifications of $\M_{1,n}$ for every deformation-open class of genus-one Gorenstein singularities.

Since one cannot use stability conditions on prestable curves to pick out arbitrary deformation-open classes of singularities, let us consider the weaker question: Does every curve singularity appear on some stable modular compactification of $\M_{g,n}$ for suitable $g$ and $n$? Surprisingly, the answer is `yes' if $n = 1$, but `no' if $n=0$. In fact, the following corollary shows that a ramphoid cusp $(y^2=x^5)$ can never arise in a stable modular compactification of $\M_{g}$.

\begin{corollary}\label{C:2}
\begin{enumerate}
\item[]
\item Every smoothable curve singularity arises in some stable modular compactification of $\M_{g,1}$ for $g>>0$. 
\begin{comment}
\item Every singularity of type $(h,m)$ arises in some stable modular compactification of $\M_{h,n}$ for $n>>0$.
\end{comment}
\item No singularity of genus $\geq 2$ arises in any stable modular compactification of $\M_{g}$.
\end{enumerate}
\end{corollary}
\begin{proof}
For (1), see Example \ref{E:CrazyAssignment}. For (2), it suffices to prove that an extremal assignment $\Z$ over $\SM_{g}$ can never pick out a genus two subcurve. If $C^{s}$ is an unmarked stable curve and $Z \subset \Z(C^{s})$ is a connected component of genus two, we obtain a contradiction as follows: first, specialize $C^{s}$ so that $Z$ splits off an elliptic bridge. Second, smooth all nodes external to the elliptic bridge. Finally, specialize to a ring of elliptic bridges (See Figure \ref{F:NoGenus2Tail}). Applying axiom 3 of Definition \ref{D:Assignment} to this sequence of specializations, we conclude that if $D^{s}$ is a ring of $g-1$ elliptic bridges, then $\Z(D^{s}) \subset D^{s}$ is non-empty. If $G$ is the dual graph of $D^{s}$, then $\Aut(G)$ acts transitively on the vertices of $G$, so axiom 2 implies that $\Z(C^{s})=C^{s}$. But this contradicts axiom 1. We conclude that an extremal assignment over $\SM_{g}$ can never pick out a genus two subcurve.
\end{proof}

\begin{figure}
\scalebox{.50}{\includegraphics{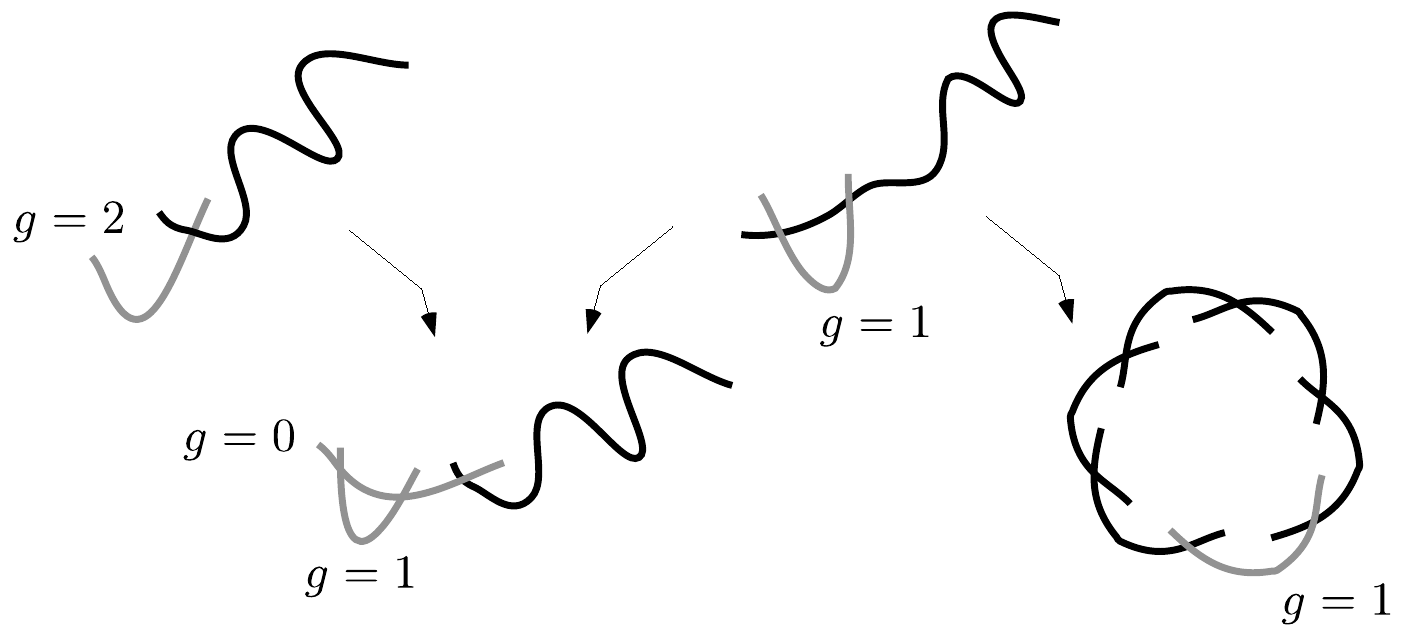}}
\caption{Sequence of specializations showing that any extremal assignment which picks out a genus two tail must also pick out an elliptic bridge within a ring of elliptic bridges.}\label{F:NoGenus2Tail}
\end{figure}

\subsubsection{Modular Compactifications of $M_{0,n}$}\label{SS:GenusZero}
In the previous section, we saw that $\Z$-stability does not give an entirely satisfactory theory of stability conditions for curves. In this section, we will see that $\Z$-stability \emph{does} give a satisfactory theory of stability conditions when $g=0$. In particular, we will see that every modular compactification of $M_{0,n}$ is automatically stable, so Theorem \ref{T:Main} actually classifies \emph{all} modular compactifications of $M_{0,n}$.

The starting point of our analysis is the following classification of genus zero singularities. It turns out that any genus zero singularity with $m$ branches is analytically isomorphic to the union of $m$ coordinate axes in $\mathbb{A}^{m}$, and we call such singularities \emph{rational $m$-fold points}.

\begin{definition}[Rational $m$-fold point]
Let $C$ be a curve over an algebraically closed field $k$. We say that $p \in C$ is a \emph{rational $m$-fold point} if
$$\hat{O}_{C,p} \simeq k[[x_1, \ldots, x_m]]/(x_ix_j: 1 \leq i<j \leq m).
$$
\end{definition}

\begin{lemma}\label{P:GenusZeroSingularities}
\begin{itemize}
\item[]
\end{itemize}
\begin{enumerate}
\item If $p \in C$ is a singularity with genus zero and $m$ branches, then $p$ is a rational $m$-fold point.
\item The rational $m$-fold point is smoothable.
\end{enumerate}
\end{lemma}
\begin{proof}
(1) is elementary. For (2), one can realize a smoothing of the rational $m$-fold point by taking a pencil of hyperplane sections of the cone over the rational normal curve of degree $m$. Both statements are proved in \cite{Stevens}.
\end{proof}
\begin{comment}
\begin{corollary}\label{C:Treelike}
Suppose that $C$ is  a reduced curve of arithmetic genus zero. Then any singular point $p \in C$ is a rational $m$-fold point (for some integer $m$), and the normalization of $C$ at $p$ has $m$ connected components.
\end{corollary}
\begin{proof}
Suppose $p \in C$ is a singular point with $m$ branches, and let $C_{1}, \ldots, C_{l}$ be the connected components of the normalization of $C$ at $p$. Then we have
$$
0=p_{a}(C)=\sum_{i=1}^{l}p_{a}(C_i)+\delta(p)-l+1.
$$
Since $p_{a}(C_i) \geq 0$, $l \leq m$, and $\delta(p) \geq m-1$, we conclude that $l=m$ and $\delta(p)=m-1$.
\end{proof}
\end{comment}

\begin{corollary}\label{C:GenusZeroSmoothability}
Every reduced connected curve of arithmetic genus zero is smoothable, i.e. $\U_{0,n} = \V_{0,n}$.
\end{corollary}
\begin{proof}
A complete reduced curve is smoothable iff its singularities are smoothable (I.6.10, \cite{Kollar}). The only singularities on a reduced curve of arithmetic genus zero curve are rational $m$-fold points, so all such curves are smoothable.
\end{proof}

Next, we study automorphisms of genus zero singular curves. If $(C, \pn)$ is an $n$-pointed curve of arithmetic genus zero over an algebraically closed field $k$, it is convenient to define $$\Aut^{0}_k(C, \pn) \subset \Aut_k(C, \pn)$$ to be the subgroup of the automorphisms which fix each component and each singular point of $C$. Then we have

\begin{lemma}\label{L:Aut}
Let $(C, \pn)$ be an $n$-pointed curve of arithmetic genus zero, and let $\pi: \tilde{C} \rightarrow C$ be the normalization of $C$. Let $\{\tilde{p}_i\}_{i=1}^{n}$ be the points of $\tilde{C}$ lying above $\pn$, and let $\{\tilde{q}_i\}_{i=1}^{m}$ be the points lying above the singular locus of $C$, and consider $(\tilde{C}, \{\tilde{p}_i\}_{i=1}^{n}, \{\tilde{q}_i\}_{i=1}^{m})$ as an $n+m$ pointed curve. Then the natural map
$$
\Aut^0_k(C, \pn) \hookrightarrow \Aut^0_k(\tilde{C}, \{\tilde{p}_i\}_{i=1}^{n}, \{\tilde{q}_i\}_{i=1}^{m})
$$
is an isomorphism.
\end{lemma}
\begin{proof}
Clearly, an automorphism $\phi \in \Aut^0_k(C, \pn)$ induces an automorphism $\tilde{\phi} \in  \Aut^0_k(\tilde{C}, \{\tilde{p}_i\}_{i=1}^{n}, \{\tilde{q}_i\}_{i=1}^{m})$. Conversely, an automorphism $\tilde{\phi} \in  \Aut^0_k(\tilde{C}, \{\tilde{p}_i\}_{i=1}^{n}, \{\tilde{q}_i\}_{i=1}^{m})$ descends to an automorphism of $(C, \pn)$ iff the natural map
$$
\tilde{\phi}^*:\O_{\tilde{C}} \simeq \O_{\tilde{C}}
$$
preserves the subsheaf of functions pulled-back from $C$, i.e. if $\phi^*(\pi^*\O_{C})=\pi^*\O_{C}$. Since the only singularities of $C$ are rational $m$-fold points, $\pi^*\O_{C} \subset \O_{\tilde{C}}$ is simply the $k$-subalgebra generated by all functions vanishing at $\{\tilde{q}_i\}_{i=1}^{m}$, and this is clearly preserved.
\end{proof}

\begin{corollary}\label{C:GenusZeroStability} Every modular compactification of $M_{0,n}$ is stable.
\end{corollary}
\begin{proof}
Let $\X$ be a modular compactification of $M_{0,n}$, and let $[C,\{p_i\}_{i=1}^{n}] \in \X$ be a geometric point over an algebraically closed field $k$. Since $\X$ is proper over $\Spec \mathbb{Z}$, the automorphism group $\Aut_{k}(C,\{p_i\}_{i=1}^{n})$ must be proper over $k$.

If $\tilde{C}$ contains an irreducible component with one or two distinguished points, then Lemma \ref{L:Aut} implies $\Aut^0_{k}(C,\{p_i\}_{i=1}^{n})$ contains a factor which is isomorphic to $\Aut_k(\P^{1},\infty)$ or $\Aut_k(\P^{1},0,\infty)$, neither of which is proper. We conclude that each irreducible component of $\tilde{C}$ must have at least three distinguished points.
\end{proof}

\begin{comment}
\begin{corollary}
Every modular compactification of $\overline{M}_{0,n}$ is an algebraic space.
\end{corollary}
\begin{proof}
By \cite{LMB}, it suffices to show that the stablizers of every geometric point are trivial. But, by the previous corollary, each irreducible component of $\tilde{C}$ has at least three distinguished points, which implies $\Aut(C)$ has no $k$-points. Need some combinatorics associated to dual graphs.
\end{proof}
\end{comment}
In light of these remarks, we obtain the following corollary of our main result.
\begin{theorem}\label{T:GenusZero}
\begin{itemize}
\item[]
\item[(1)] If $\X \subset \U_{0,n}$ is any open proper substack, then $\X=\SM_{0,n}(\Z)$ for some extremal assignment $\Z$.
\item[(2)] $\X$ is an algebraic space.
\end{itemize}
\end{theorem}
\begin{proof}
(1) follows from Corollary \ref{C:GenusZeroStability}, Corollary \ref{C:GenusZeroSmoothability}, and Theorem \ref{T:Main}. For (2), it suffices to show that if $[C, \pn] \in \SMzero(\Z)$ is any geometric point, then $\Aut_{k}(C, \pn)$ is trivial. Since every component of $\tilde{C}$ has  at least three distinguished points, we have $\Aut_{k}^{0}(C,\pn)=\{0\}$, so we only need to see that every automorphism of a prestable genus zero curve fixes the irreducible components and singular points. This is an elementary combinatorial consequence of the fact that every component of $(C,\pn)$ has at least three distinguished points.
\end{proof}

\subsection{Outline of proof}\label{S:Outline}
In this section, we give a detailed outline of the proof of Theorem \ref{T:Main}, which occupies sections 2-4 of this paper.

In Section \ref{S:Preliminaries}, we establish several fundamental lemmas, which are used repeatedly throughout. In Section \ref{S:ExtendingFamilies}, we prove that a birational map between two generically-smooth families of curves over a normal base is automatically Stein (Lemma \ref{L:Normality}). We also prove that, after an alteration of the base, one can birationally dominate any family of prestable curves by a family of stable curves (Lemma \ref{L:ExtendingCurves}). Taken together, these lemmas allow us to analyze deformations and specializations of prestable curves by studying the deformations and specializations of the stable curves lying over them. 

In Section \ref{S:BirationalMaps}, we define a \emph{contraction morphism of curves} to be a surjective morphism with connected fibers, which contracts subcurves of genus $g$ to singularities of genus $g$. The motivation for this definition is Lemma \ref{L:BirationalBaseChange}, which says that a birational contraction $\C_{1} \rightarrow \C_{2}$ between two irreducible families of generically smooth curves induces a contraction of curves on each geometric fiber. Finally, in Section \ref{S:ZStability}, we define the stability condition associated to an extremal assignment $\Z$: An $n$-pointed curve is \emph{$\Z$-stable} if there exists a stable curve $(C^s, \spn)$ and a contraction $\phi:(C^s, \spn) \rightarrow (C, \pn)$ with $\Exc(\phi)=\Z(C^s)$. An important consequence of the axioms for an extremal assignment is that the existence of a single contraction $\phi:(C^s, \spn) \rightarrow (C, \pn)$ with $\Exc(\phi)=\Z(C^s)$ implies that $\Exc(\phi)=\Z(C^s)$ for \emph{any} contraction from a stable curve (Corollary \ref{C:Independence}).

In Section \ref{S:Openness}, we prove that the locus of $\Z$-stable curves is open in $\V_{g,n}$, the main component of the moduli stack of all curves. Given a generically-smooth family of curves $(\C \rightarrow T, \sigman)$ over an irreducible base $T$, we must show that the set
\[
S:=\{t \in T \, | \, (\C_{\overline{t}}, \{\sigma_i(\overline{t})\}_{i=1}^{n})\text{ is $\Z$-stable} \}
\]
is open in $T$. It is sufficient to prove that $i^{-1}(S)$ is open after any proper surjective base-change $i:\tilde{T} \rightarrow T$. Thus, using the results of Section 2.1, we may assume there exists a stable curve over $T$ birationally dominating $\C$, i.e. we have a diagram
\[
\xymatrix{
\C^{s} \ar[rr]^{\phi} \ar[dr]^{\pi^s} && \C \ar[dl]_{\pi}\\
&T  \ar@/^1pc/[lu]^{\{\sigma^s_i\}_{i=1}^{n}} \ar@/_1pc/[ru]_{\sigman}&
}
\]
By Section 2.2, the fibers of $\phi$ are contractions of curves. Thus, the fiber $\pi^{-1}(t)$ is $\Z$-stable if and only if $\Exc(\phi_t)=\Z(C^{s}_t)$. Thus, it suffices to prove that $$\{t \in T \, |\, \Exc(\phi_{\overline{t}})=\Z(C_{\overline{t}})\}$$
is open in $T$. This is an immediate consequence of axiom 3 in the definition of an extremal assignment.

In Section \ref{S:Properness}, we prove that $\Z$-stable curves satisfy the unique limit property. To prove that $\Z$-stable limits exist, we use the classical stable reduction theorem and Artin's criterion for the contractibility of 1-cycles on a surface. Given a family of smooth curves over the punctured disc $\Delta^*$, we may complete it to a stable curve $\C^{s} \rightarrow \Delta$. Using Artin's criterion, we construct a birational morphism $\phi:\C^{s} \rightarrow \C$ with $\Exc(\phi)=\Z(C^{s})$, where $C^{s} \subset \C^{s}$ is the special fiber. The restriction of $\phi$ to the special fiber induces a contraction of curves $\phi_0:C^{s} \rightarrow C$ with $\Exc(\phi_0)=\Z(C^{s})$. Thus, $C$ is the desired $\Z$-stable limit.

To prove that $\Z$-stable limits are unique, we show that if $\C_{1} \rightarrow \Delta$ and $\C_{2} \rightarrow \Delta$ are two $\Z$-stable families with smooth isomorphic generic fiber, then there exists a stable curve $\C^{s} \rightarrow \Delta$ and birational maps
\[
\xymatrix{
&\C \ar[dr]^{\phi_2} \ar[ld]_{\phi_1}&\\
\C_1&&\C_2\\
}
\]
Since $\phi_1$ and $\phi_2$ induce contraction morphisms on the special fiber, the hypothesis that $\C_{1}$ and $\C_{2}$ are $\Z$-stable implies that $\Exc(\phi_1)=\Z(C^{s})$ and $\Exc(\phi_2)=\Z(C^s)$. In particular, $\Exc(\phi_1)=\Exc(\phi_2)$. Since $\phi_1$ and $\phi_2$ are Stein morphisms,  we conclude that the rational map $\C_1 \dashrightarrow \C_2$ extends to an isomorphism.

Section \ref{S:Classification} is devoted to the proof that any stable modular compactification $\X \subset \V_{g,n}$ takes the form $\X=\SM_{g,n}(\Z)$ for some extremal assignment $\Z$ over $\SM_{g,n}$. Given a stable modular compactification $\X \subset \SM_{g,n}(\Z)$, Lemma \ref{L:Diagram} produces a diagram 
\[
\xymatrix{
\C^{s} \ar[dr]^{\pi^{s}} \ar[rr]^{\phi}&&\C \ar[dl]_{\pi}\\
&T \ar[dl]_{p} \ar[dr]^{q} \ar@/_1pc/[ru]_{\sigman}  \ar@/^1pc/[lu]^{\{\sigma_i^{s} \}_{i=1}^{n}}&\\
\SM_{g,n}&\U \ar@{^{(}->}[r] \ar@{_{(}->}[l]&\X\\
}
\]
satisfying
\begin{itemize}
\item[(0)] $\U \subset \M_{g,n}$ is an open dense substack,
\item[(1)] $T$ is a normal scheme,
\item[(2)] $p$ and $q$ are representable proper dominant generically-\'{e}tale morphisms,
\item[(3)] $\pi^s$ and $\pi$ are the families induced by $p$ and $q$ respectively,
\item[(4)] $\phi$ is a birational morphism.
\end{itemize}
For any graph $G$,  set $T_{G}:=\M_{G} \times_{\SM_{g,n}} T$, i.e. $T_{G}$ is the locally-closed subscheme over which the fibers of $\pi^{s}$ have dual graph isomorphic to $G$. In addition, for any $t \in T$, let $G_{t}$ denote the dual graph of the fiber $(\pi^{s})^{-1}(t)$. We wish to associate to $\X$ an extremal assignment $\Z$ by setting
\[
\Z(G):=i(\Exc(\phi_t)) \subset G,
\]
for some choice of $t \in T_{G}$ and some choice of isomorphism $i:G_{t} \simeq G$. The key point is to show that the subgraph $\Z(G) \subset G$ does not depend on these choices (Proposition \ref{P:ExceptionalLocus}). We then show that $\Z$ satisfies axioms 1-3 in Definition \ref{D:Assignment}. Axiom 1 is an immediate consequence of the fact that $\phi$ cannot contract an entire fiber of $\pi^s$. Axiom 2 is forced by the separatedness of $\X$. For axiom 3, consider a one-parameter family of stable curves $(\C^{s} \rightarrow \Delta, \sigmans)$ inducing a specialization of dual graphs $G \leadsto G'$. Since $T \rightarrow \SM_{g,n}$ is proper, we may lift the natural map $\Delta \rightarrow \SM_{g,n}$ to $T$, and consider the induced birational morphism of families over $\Delta$:
\[
\xymatrix{
\C^{s}  \ar[dr] \ar[rr]^{\phi} && \C \ar[dl]\\
&\Delta&\\
}
\]
After a finite base-change, we may assume that $\C^{s}=\C_{1} \cup \ldots \cup \C_{m}$, where each $\C_{i} \rightarrow \Delta$ is a flat family of curves with smooth generic fiber, and axiom 3 follows from the fact that $$(\C_{i})_{\bar{\eta}} \in \Exc(\phi_{\overline{\eta}}) \iff (\C_{i})_{0} \in \Exc(\phi_0).$$

Once we have established that $\Z$ is a well-defined extremal assignment, the fact that $\phi$ induces a contraction of curves over each geometric point $t \in T$ implies that each fiber $\pi^{-1}(t)$ is $\Z$-stable for this assignment. Since $T \rightarrow \X$ is surjective, we conclude that each geometric point of $\X$ corresponds to a $\Z$-stable curve. Thus, the open immersion $\X \hookrightarrow \V_{g,n}$ factors through $\SM_{g,n}(\Z)$. The induced map $\X \hookrightarrow \SM_{g,n}(\Z)$ is proper and dominant, so $\X=\SM_{g,n}(\Z)$ as desired.

\subsection{Notation}\label{S:Notation}
The following notation will be in force throughout: An \textit{$n$-pointed curve} consists of a pair $(C, \pn)$, where $C$ is a reduced connected complete one-dimensional scheme of finite-type over an algebraically closed field, and $\pn$ is an ordered set of $n$ marked points of $C$. The marked points need not be smooth nor distinct. We say that a point on $C$ is \textit{distinguished} if it is marked or singular, and that a point on the normalization $\tilde{C}$ is \textit{distinguished} if it lies above a distinguished point of $C$. A curve $(C,\pn)$ is \textit{prestable} (resp. \textit{presemistable}) if every rational component of $\tilde{C}$ has at least three (resp. two) distinguished points. A curve $(C, \pn)$ is \emph{smooth} if $C$ is smooth and the points $\pn$ are distinct. A curve $(C,\pn)$ is \textit{nodal} if the only singularities of $C$ are ordinary nodes and the marked points of $C$ are smooth and distinct. We say that $(C,\pn)$ is \textit{stable} (resp. \textit{semistable}) if $(C,\pn)$ is nodal and prestable (resp. presemistable).

All these definitions extend to general bases in the usual way: Given a scheme $T$, an \textit{$n$-pointed curve over $T$} consists of a flat proper finitely-presented morphism $\pi: \C \rightarrow T$, together with a collection of sections $\sigman$, such that the geometric fibers are $n$-pointed curves. We say that a curve over $T$ is \textit{prestable}, \textit{presemistable}, \textit{nodal}, \textit{stable}, or \textit{semistable} if the corresponding conditions hold on geometric fibers. Families are typically denoted in script, while geometric fibers are denoted in regular font. Note that we always allow the total space of a family of curves to be an algebraic space. Whenever we consider a morphism $\phi:\C_{1} \rightarrow \C_{2}$ between two families of $n$-pointed curves, say $(\C_{1} \rightarrow T, \sigman)$ and $(\C_{2} \rightarrow T, \taun)$, we always assume that $\phi \circ \sigma_{i}=\tau_{i}$ for all $i$.

$\Delta$ will always denote the spectrum of a discrete valuation ring $R$ with algebraically closed residue field $k$ and field of fractions $K$. We make constant use of the fact that if $x,y \in X$ are two points of any noetherian scheme (or Deligne-Mumford stack) with $y \in \overline{\{ x\}}$, then there exists a map $\Delta \rightarrow X$ sending $\eta \rightarrow x$, $0 \rightarrow y$. When we speak of a finite base-change $\Delta' \rightarrow \Delta$, we mean that $\Delta'$ is the spectrum of a discrete valuation ring $R' \supset R$ with field of fractions $K'$, where $K' \supset K$ is a finite separable extension. We use the notation
\begin{align*}
0&:=\Spec k \rightarrow \Delta,\\
\eta&:=\Spec K \rightarrow \Delta,\\
\overline{\eta}&:=\Spec \overline{K} \rightarrow \Delta,
\end{align*}
for the closed point, generic point, and geometric generic point respectively. Thus, $C_0, \C_{\eta}, C_{\overline{\eta}}\,$ and $C'_0, \C'_{\eta}, C'_{\overline{\eta}}\,$  denote the special fiber, generic fiber, and geometric generic fibers of $\C \rightarrow \Delta$ and $\C' \rightarrow \Delta$ respectively. Sometimes we omit the subscript `0', and denote the special fibers of $\C \rightarrow \Delta$ or $\C' \rightarrow \Delta$ by $C$ or $C'$. Also, we let $\eta_{\X}$ denote the generic point of any irreducible stack or scheme $\X$. 

$\SM_{g,n}$ will denote the moduli stack (over $\Spec \mathbb{Z}$) of $n$-pointed stable curves of genus $g$. Recall that $\SM_{g,n}$ admits a stratification by topological type, i.e.
$$
\SM_{g,n} = \coprod_{G} \M_{G},
$$
where the union runs over all isomorphism classes of dual graphs of $n$-pointed stable curves of genus $g$, and $\M_{G} \subset \SM_{g,n}$ is the locally-closed substack parametrizing stable curves whose dual graph is isomorphic to $G$.

While the language of stacks is employed throughout, everything we do is essentially topological. We use little more than the definition of the Zariski topology for an algebraic stack, and the various valuative criteria for specialization and properness, for which we refer the reader to \cite{LMB}.\\

\textbf{Acknowledgements.}
Joe Harris offered countless suggestions and insights throughout the course of this project, and it is a pleasure to acknowledge his great influence. I would also like to thank Jarod Alper, Maksym Fedorchuk, Jack Hall, Brendan Hassett, Sean Keel, Matthew Simpson, and Fred van der Wyck, for many helpful conversations, by turns mathematical, whimsical, or therapeutic. I am especially indebted to Fred van der Wyck, whose work on moduli of crimping data hovers in the background (carefully hidden!) of several parts of this paper.

\section{Preliminaries on $\Z$-stability} \label{S:Preliminaries}
\subsection{Extending families of prestable curves}\label{S:ExtendingFamilies}
In this section, we present two key lemmas, which will be used repeatedly. Lemma \ref{L:Normality} says that a birational map between two generically-smooth families of curves over a normal base is automatically Stein. Lemma \ref{L:ExtendingCurves} says that, after an alteration of the base, one can dominate any family of prestable curves by a family of stable curves. (Recall that an alteration is proper, surjective, generically-\'{e}tale morphism.) Taken together, these two lemmas allow us to reduce questions about families of prestable curves to questions about stable curves.

\begin{lemma}[Normality of generically-smooth families of curves]\label{L:Normality}
\hfill
\begin{enumerate}
\item Suppose that $S$ is an irreducible, normal, noetherian scheme, and that $\C \rightarrow S$ is a curve over $S$ with smooth generic fiber. Then $\C$ is normal.
\item Suppose that $S$ is an irreducible, normal, noetherian scheme, and that $\C_{1} \rightarrow S$ and $\C_{2} \rightarrow S$ are curves over $S$ with smooth generic fiber. If $\phi:\C_{1} \rightarrow \C_{2}$ is a birational morphism over $S$, then $\phi_{*}\O_{\C_1}=\O_{\C_2}$.
\end{enumerate}
\end{lemma}
\begin{proof}
For (1), first observe that since $\C \rightarrow S$ is smooth in the generic fiber and has isolated singularities in every fiber, $\C$ must be regular in codimension one. Furthermore, since $\C \rightarrow S$ is a flat morphism with both base and fibers satisfying Serre's condition $S_2$, $\C$ satisfies $S_{2}$ as well \cite[6.4.2]{EGAIV}. By Serre's criterion, $\C$ is normal.

For (2), $\phi: \C_{1} \rightarrow \C_{2}$ is a proper birational morphism of normal noetherian algebraic spaces. Since a finite birational morphism of normal algebraic spaces is an isomorphism \cite[4.7]{Knutson}, $\phi$ is equal to its own Stein factorization, i.e. $\phi_{*}\O_{\C_1}=\O_{\C_2}$.
\end{proof}
\begin{lemma}[Extending prestable curves to stable curves]\label{L:ExtendingCurves}
\hfill
\begin{itemize}
\item[(1)] Let $T$ be an integral noetherian scheme, and $(\C \rightarrow T, \sigman)$ an $n$-pointed curve over $T$ with smooth generic fiber. There exists an alteration $\tilde{T} \rightarrow T$, and a diagram
\[
\xymatrix{
\C^{s} \ar@{-->}[rr]^{\phi} \ar[dr] && \tilde{\C} \ar[dl]\\
&\tilde{T}  \ar@/^1pc/[lu]^{\{\sigma^s_i\}_{i=1}^{n}} \ar@/_1pc/[ru]_{\{\tilde{\sigma}_i\}_{i=1}^{n}}&
}
\]
where $(\C^{s} \rightarrow \tilde{T}, \{\sigma^s_i\}_{i=1}^{n})$ is a stable curve, $(\tilde{\C} \rightarrow \tilde{T},\{\tilde{\sigma}_i\}_{i=1}^{n})$ is the $n$-pointed curve induced by base-change,  and $\phi$ is a birational map over $\tilde{T}$.
\item[(2)] We may choose the alteration $\tilde{T} \rightarrow T$ so that $\tilde{T}$ is normal, and the open subset $S \subset \tilde{T}$ defined by
$$
S:=\{t \in \tilde{T} \,| \, \text{$\phi$ is regular in a neighborhood of the fiber $\C_{t}^{s}$} \}
$$
contains every geometric point $t \in \tilde{T}$ such that the fiber $(\tilde{\C}_{t}, \{\tilde{\sigma}_i(t)\}_{i=1}^{n})$ is prestable.
\end{itemize}
\end{lemma}
\begin{proof}
The moduli stack $\SM_{g,n}$ admits a finite generically-etale cover by a scheme, say $M \rightarrow \SM_{g,n}$ (\cite{DeJong1}, 2.24). Let
$$
U:=\{t \in T \,| \, \text{ $(\C_{t}^{s}, \{\sigma_i(t)\}_{i=1}^{n})$ is stable.}\},
$$
and consider the Cartesian diagram
\[
\xymatrix{
U \times_{\SM_{g,n}} M \ar[d] \ar[r]&M \ar[d]\\
U \ar[r]& \SM_{g,n}\\
}
\]
Let $\tilde{U}$ be any irreducible component of $U \times_{\SM_{g,n}} M$ dominating $U$, and define $\tilde{T}$ to be the closure of the image of $\tilde{U}$ in $T \times_{\Spec \mathbb{Z}}M.$ Then $\tilde{T} \rightarrow T$ is an alteration satisfying the conclusion of (1).

Next, we claim that we may choose the alteration $\tilde{T} \rightarrow T$, so that there exists a diagram

\[
\xymatrix{
&\C^{n} \ar[rd]^{\phi_2} \ar[ld]_{\phi_1}&\\
\C^{s} \ar@{-->}[rr]^{\phi}\ar[dr]^{\pi_1} &&\tilde{\C} \ar[dl]_{\pi_2}\\
&\tilde{T} \ar@/^1pc/[lu]^{\{\sigma^s_i\}_{i=1}^{n}} \ar@/_1pc/[ru]_{\{\tilde{\sigma}_i\}_{i=1}^{n}}&
}
\]
satisfying
\begin{enumerate}
\item $\tilde{T}$ is a normal noetherian scheme,
\item $(\C^{n} \rightarrow S, \{\tau_{i}\}_{i=1}^{n})$ is a nodal curve,
\item $\phi_1$ and $\phi_2$ are regular birational maps over $\tilde{T}$.
\end{enumerate}

To see this, start by taking $\tilde{T} \rightarrow T$ as in (1). After blowing-up $\tilde{T}$ further, we may assume that there exists a flat projective morphism $X \rightarrow \tilde{T}$ of relative dimension one, admitting regular birational maps to both $\C^{s}$ and $\tilde{\C}$. (Apply Chow's lemma and the flattening results of \cite{RG} to the graph of $\phi$.) Let $Z \subset X$ denote the pure codimension-one subscheme obtained by taking the strict transform of the sections $\{\tilde{\sigma}_i\}_{i=1}^{n}$ on $X$. By a theorem of de Jong (\cite{DeJong2}, 2.4), we may alter $(X \rightarrow \tilde{T}, Z)$ to a nodal curve, i.e. there exists an alteration $\tilde{T}' \rightarrow \tilde{T}$ with $\tilde{T}'$ a normal noetherian scheme, a nodal curve $(\C^{n} \rightarrow \tilde{T}', \{\tau_{i}\}_{i=1}^{n})$, and a commutative diagram
\[
\xymatrix{
\C^{n} \ar[d] \ar[r]& X \ar[d]\\
\tilde{T}' \ar[r]& \tilde{T}\\
}
\]
such that the induced map $(\C^{n}, \cup_{i=1}^{n}\tau_{i}) \rightarrow (X \times_{\tilde{T}} \tilde{T}', Z \times_{\tilde{T}} \tilde{T}')$ is an isomorphism over the generic point of $\tilde{T}'$. In particular, $\C^{n}$ admits regular birational maps to both $(\C^{s} \times_{\tilde{T}} \tilde{T}')$ and $\tilde{\C} \times_{\tilde{T}} \tilde{T}'$, so $\tilde{T}' \rightarrow T$ gives the desired alteration.

Now fix an alteration $\tilde{T} \rightarrow T$ and a diagram satisfying (1)-(3) above. Since $\pi_1$ is proper, the set
\[
S:=\{t \in \tilde{T} \,| \, \text{$\phi$ is regular in a neighborhood of the fiber $\C_{t}^{s}$} \}
\]
is open in $\tilde{T}$. We must show that if $t \in \tilde{T}$ is any point such that the fiber $(\tilde{C}_{t}, \{\tilde{\sigma}_i(t)\}_{i=1}^{n})$ is prestable, then $t \in S$. By Lemma \ref{L:Normality}, we have $(\phi_1)_*\O_{\C^{n}}=\O_{\C^s}$ and $(\phi_2)_*\O_{\C^{n}}=\O_{\tilde{\C}}$. Thus, it suffices to show that if $E \subset \C^{n}_t$ is any irreducible component contracted by $\phi_1$, then $E$ is also contracted by $\phi_2$. By the uniqueness of stable reduction, we have
$$
\Exc(\phi_{1})_{t}=\{ E \subset \C^{n}_t |\text{ $E$ is smooth rational with one or two distinguished points }\}
$$
Thus, if $E \subset \Exc(\phi_{1})_{t}$ is not contracted by $\phi_{2}$, its image is a rational component of $(\tilde{C}_{t}, \{\tilde{\sigma}_i(t)\}_{i=1}^{n})$ with fewer than three distinguished points. This is a contradiction, since $(\tilde{C}_{t}, \{\tilde{\sigma}_i(t)\}_{i=1}^{n})$ is prestable.
\end{proof}

\subsection{Contractions of curves}\label{S:BirationalMaps}
\begin{comment}
Suppose that $\C^{s} \rightarrow \Delta$ is a generically-smooth family of stable curves over a disc, and that $Z$ is an arbitrary subcurve of the special fiber. In this section, we explain how to produce birational contractions
\[
\xymatrix{
\C^{s} \ar[rr]^{\phi} \ar[dr]&&\C \ar[dl]\\
&\Delta&
}
\]
such that $\Exc(\phi)=Z$ and $\phi$ replaces each connected component of $Z$ by an isolated curve singularity in the special fiber of $\C$.
\end{comment}

\begin{comment}
\begin{example}
It is not difficult to see that
\begin{align*}
g(p)=0, m(p)=1 &\iff p \in C\text{ is smooth},\\
g(p)=0, m(p)=2 &\iff p \in C\text{ is an ordinary node }(y^2-x^2),\\
g(p)=1, m(p)=1 &\iff p \in C\text{ is an ordinary cusp }(y^2-x^3).\
\end{align*}
By contrast, there are two analytic isomorphism classes of singularities with $g(p)=1, m(p)=2$, namely the tacnode $(y^2-x^4)$ and the spatial singularity obtained as an ordinary cusp with a smooth transverse branch $(y^2-x^3,xz,yz)$. This assertion is proved in Proposition \ref{P:Classification}.
\end{example}
\end{comment}

In Lemma \ref{L:BirationalBaseChange}, we will see that birational contractions between generically-smooth families of curves have the effect of replacing arithmetic genus $g$ subcurves by isolated singularities of genus $g$. This motivates the following definition.

\begin{definition}[Contraction of curves]\label{D:BirationalMap}
If $\phi:C \rightarrow D$ is a morphism of curves, let $\Exc(\phi)$ denote the union of those irreducible components $E \subset C$ which are contracted to a point $\phi(E) \in D$. We say that $\phi$ is a \emph{contraction} if it satisfies
\begin{itemize}
\item[(1)] $\phi$ is surjective with connected fibers,
\item[(2)] $\phi$ is an isomorphism on $C - \Exc(\phi)$,
\item[(3)] If $Z$ is any connected component of $\Exc(\phi)$, then the point $p:=\phi(Z) \in D$ satisfies
$g(p)=p_a(Z)$ and $m(p)=|Z \cap Z^{c}|$, where $g(p)$ and $m(p)$ are the genus and number of branches of $p$, as in Definition \ref{D:Genus}.

\end{itemize}
\end{definition}

\begin{remark}\label{R:Normalize}
 If $C$ is a nodal curve and $C \rightarrow D$ is a contraction, then we have a decomposition
$$
C=\tilde{D} \cup Z_1 \cup \ldots \cup Z_k,$$
where $Z_1, \ldots, Z_k$ are the connected components of $\Exc(\phi)$, and $\tilde{D}$ is the normalization of $D$ at $\phi(Z_1), \ldots, \phi(Z_k) \in C$. This is immediate from the fact that if $C$ is nodal, the points of $\overline{C \backslash Z_i}$ lying above $\phi(Z_i) \in D$ are smooth.

\end{remark}

\begin{example}[Contraction morphisms contracting an elliptic bridge]\label{E:Birat}
Let $$C=C_{1} \cup E \cup C_{2}$$ be a nodal curve with an elliptic bridge (see Figure \ref{F:BirationalMap}). Then there exist contraction morphisms contracting $E$ to a tacnode $(y^2-x^4)$ or to a planar cusp with a smooth transerverse branch $(xz, yz,y^2-x^3)$, since both these singularities have two branches and genus one. By contrast, the map contracting $E$ to an ordinary node is not a contraction because the genus of an ordinary node is zero. In fact, the map contracting $E$ to a node is the Stein factorization of the given contractions.
\end{example}
\begin{figure}
\scalebox{.50}{\includegraphics{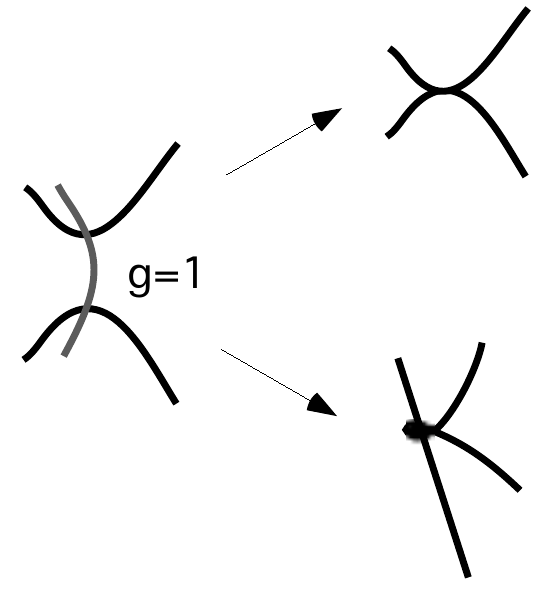}}
\caption{Two contractions of curves, each contracting an elliptic bridge.}\label{F:BirationalMap}
\end{figure}

\begin{proposition}[Existence of Contractions]\label{P:Contractions}
Let $\Delta$ be the spectrum of a discrete valuation ring, and $\pi:\C^{n} \rightarrow \Delta$ a generically-smooth, nodal curve over $\Delta$.  If $Z \subsetneq C^{n}$ is a proper subcurve of the special fiber, then there exists a diagram
\[
\xymatrix{
\C^{n} \ar[rr]^{\phi} \ar[dr]&&\C \ar[dl]\\
&\Delta&
}
\]
such that
\begin{itemize}
\item[(1)] $\phi$ is proper, birational, $\phi_*\O_{\C^{n}}=\O_{\C}$, and $\Exc(\phi)=Z$.
\item[(2)] $\C \rightarrow \Delta$ is a flat family of geometrically reduced connected curves.
\item[(3)] The restriction of $\phi$ to the special fiber induces a contraction of curves.
\end{itemize}
\end{proposition}
\begin{proof}
We claim that it is sufficient to produce a birational morphism $\phi:\C^{n} \rightarrow \C$ such that $\Exc(\phi)=Z$. Indeed, after taking the Stein factorization, we may assume that $\phi$ satisfies (1). By Lemma \ref{L:Normality}, $\C^{n}$ is normal, so $\C$ is as well. In particular, the special fiber $C$ is Cohen-Macaulay and therefore has no embedded points. Since each component of $C$ is the birational image of an irreducible component of $C^{s}$, no component of $C$ can be generically non-reduced. It follows that the special fiber is reduced and connected. In addition, $\C \rightarrow \Delta$ is flat since the generic point of $\C$ maps to the generic point of $\Delta$. This shows that $\C \rightarrow \Delta$ satisfies (2). Finally condition (3) is a consequence of the more general statement proved in Lemma \ref{L:BirationalBaseChange} below.

It remains to show that there exists a birational morphism $\phi:\C^{n} \rightarrow \C$ with $\Exc(\phi)=Z$. There exists a minimal resolution of singularities $p: \tilde{\C}^{n} \rightarrow \C^{n}$ such that $ \tilde{\C}^{n}\rightarrow \Delta$ is still a nodal curve \cite{Lipman}, and it is sufficient to produce a birational contraction $\phi: \tilde{\C}^{n} \rightarrow \C$ with $\Exc(\phi)=p^{-1}(Z)$. Thus, we may assume that the total space $\C^{n}$ is regular to begin with. 

Now $\C^{n}$ is a regular algebraic space over an excellent Dedekind ring, so there is a necessary and sufficient condition for the existence of a contraction. If $Z_1 \cup \ldots \cup Z_k$ are the irreducible components of the $Z$, then the intersection matrix $||(Z_i.Z_j)||$ must be negative-definite \cite[6.17]{Artin} . In fact, it is easy to see that any proper subcurve of the special fiber must have negative-definite intersection matrix. 

To see this, let $C_1, \ldots, C_m$ be the irreducible components of the special fiber, and $C=\sum_{i=1}^{m}C_i$ the class of the fiber. Let $Z:=a_1C_1+\ldots+a_kC_k$ be an arbitrary cycle supported on the special fiber. We will prove by induction on $k$ that $Z^{2} \leq 0$, with equality iff $Z$ is a multiple of the entire fiber.  Let us consider first the case when $Z$ is effective. After reordering, we may assume that $a_1>a_2>\ldots>a_k>0$. Since $C^{2}=C.Z=0$, we have
\begin{align*}
Z^{2}&=(Z-a_kC)^2=\left( \sum_{i=1}^{k-1}(a_i-a_k)C_i -a_{k}\sum_{i=k+1}^{m}C_i \right)^2 \\
&=\left(\sum_{i=1}^{k-1}(a_i-a_k)C_i\right)^2-2a_{k}\left(\sum_{i=1}^{k-1}(a_i-a_k)C_i\right).\left(\sum_{i=k+1}^{m}C_i\right)+a_{k}^2\left(\sum_{i=k+1}^{m}C_i\right)^2\\
&=\left(\sum_{i=1}^{k-1}(a_i-a_k)C_i\right)^2-2a_{k}\left(\sum_{i=1}^{k-1}(a_i-a_k)C_i\right).\left(\sum_{i=k+1}^{m}C_i\right)-a_{k}^2\left(\sum_{i=1}^{k}C_i \right).\left(\sum_{i=k+1}^{m}C_i\right)\\
\end{align*}
The latter two terms are obviously non-positive and by induction the first term is non-positive as well. Furthermore, $Z^{2}=0$ iff each term is zero which evidently forces $a_1=\ldots=a_{k-1}=a_k$ and $k=m$.

If $Z$ is not effective, then we may write $Z=Z_{1}-Z_{2}$, where $Z_{1}$ and $Z_{2}$ are effective with no common components. Then we have
$$
Z^2=Z_1^2-2Z_1.Z_2+Z_2^2 \leq 0,
$$
with equality iff $Z_{1}$ and $Z_{2}$ are multiples of the entire fiber. This completes the proof.
\end{proof}

\begin{lemma}\label{L:BirationalBaseChange}
Let $S$ be an irreducible, normal, noetherian scheme, and let $\pi_1:X \rightarrow S$ and $\pi_2:Y \rightarrow S$ be two curves over $S$. Suppose that $\pi_1$ is nodal, and that $\pi_1$ and $\pi_2$ are generically smooth. If we are given a birational morphism over $S$
\[
\xymatrix{
X \ar[rr]^{\phi} \ar[rd]_{\pi_1}&& Y \ar[dl]^{\pi_2}\\
&S&
}
\]
then the induced map $\phi_{s}:X_{s} \rightarrow Y_{s}$ is a contraction, for each geometric point $s \in S$.
\end{lemma}
\begin{proof}
By Lemma \ref{L:Normality}, we have $\phi_{*}\O_{X}=\O_{Y}$. Using this, we will show that $\phi_{s}$ satisfies conditions (1)-(3) of Definition \ref{D:BirationalMap}. By Zariski's main theorem, $\phi$ has geometrically connected fibers, so $\phi_{s}$ satisfies (1). Furthermore, $\phi$ is an isomorphism when restricted to the complement of the positive-dimensinal fibers of $\phi$, so $\phi_{s}$ satisfies (2). It remains to verify that $\phi_{s}$ satisfies (3).

Without loss of generality, we may assume that $Z:=\Exc(\phi_{s})$ is connected, and we must show that $p:=\phi_{s}(Z) \in Y_{s}$ is a singularity of genus $g$. Since the number of branches of $p \in Y_{s}$ is, by definition, the number of points lying above $p$ in the normalization, we have
$$m(p)=|\overline{X_{s} \backslash Z} \cap Z|.$$
To obtain $\delta(p)=p_a(Z)+m(p)-1$, note that
\begin{align*}
\delta&=\chi(X_{s},\O_{\overline{X_{s} \backslash Z} })-\chi(Y_{s},\O_{Y_{s}})\\
&=\chi(X_{s},\O_{\overline{X_{s} \backslash Z} })-\chi(X_{s},\O_{X_{s}})\\
&=-\chi(X_{s},I_{\overline{X_{s} \backslash Z}}) .
\end{align*}
The first equality is just the definition of $\delta$ since $\overline{X_{s} \backslash Z}$ is the normalization of $Y_{s}$ at $p$. The second equality follows from the fact that $X_{s}$ and $Y_{s}$ occur in flat families with the same generic fiber, and the third equality is just the additivity of Euler characteristic on exact sequences. Since $I_{\overline{X_{s} \backslash Z}}$ is supported on $Z$, we have
$$\chi(X_{s},I_{\overline{X_{s} \backslash Z}})=\chi(Z,I_{\overline{X_{s} \backslash Z}}|_{Z})= \chi(Z,\O_{Z}(-Z \cap \overline{X_{s} \backslash Z}))=1-m(p)-p_a(Z),$$
which gives the desired equality.
\end{proof}

\subsection{$\Z$-stability}\label{S:ZStability}
In this section, we define the stability condition associated to a fixed extremal assignment $\Z$. Using the definition of a contraction (Definition \ref{D:BirationalMap}), we can recast our original definition of $\Z$-stability (Definition \ref{D:Zstable}) as follows:

\begin{definition}[$\Z$-stable curve]\label{D:ZStable}
A smoothable $n$-pointed curve $(C,\pn)$ is \emph{$\Z$-stable} if there exists a  stable curve $(C^s,\spn)$ and a contraction $\phi:(C^s, \spn) \rightarrow (C,\pn)$ such that $\Exc(\phi)=\Z(C^s)$.
\end{definition}

We will make frequent use of the following observation: If $(C, \pn)$ is $\Z$-stable, and $\phi:(C^{s}, \spn) \rightarrow (C, \pn)$ is \emph{any} contraction from a stable curve, then $\Exc(\phi)=\Z(C^{s})$. (The definition of $\Z$-stability asserts the existence of a single contraction with this property.) In order to prove this, we need the following lemma which gives an explicit description of the set of stable curves admitting contractions to a fixed prestable curve.
\begin{lemma}\label{L:Mapping}
Let $(C,\pn)$ be an $n$-pointed prestable curve, and let $z_1, \ldots, z_k \in C$ be the set of points
which satisfy one of the following conditions:
\begin{itemize}
\item[(1)] $z_{i}$ is non-nodal singularity,
\item[(2)] $z_i$ is a node, and at least one marked point is supported at $z_i$, 
\item[(3)] $z_i$ is a smooth point, and at least two marked points are supported at $z_i$.
\end{itemize}
As in Definition \ref{D:Genus}, set
$m_{i}=m(z_i)$,
$g_{i}=g(z_i)$,
and let $l_{i}$ denote the number of marked points supported at $z_{i}$. There exists a map
\[
g:=g_{(C,\{p_i\})}:\prod_{i=1}^{k} \SM_{g_i, m_i+l_i} \rightarrow \SM_{g,n}
\]
with the property that a stable curve $(C^{s}, \spn)$ admits a contraction to $(C,\pn)$ iff it lies in the image of $g$.

\end{lemma}
\begin{proof}
In order to define $g$, let us relabel the marked points of $C$:
\[
\{ p_i\}_{i=1}^{n}=\{ p_j \}_{j=1}^{r} \cup \{p_{1j}\}_{j=1}^{l_1} \cup \ldots \cup \{p_{kj} \}_{j=1}^{l_{k}},
\]
where $\{p_{ij}\}_{j=i}^{l_i}$ is the set of marked points supported at $z_{i}$, and the points $\{p_{j}\}_{j=1}^{r}$ are distinct smooth points of $C$. Let $\tilde{C} \rightarrow C$ denote the normalization of $C$ at $\{ z_i\}_{i=1}^{k}$, and let $\{\tilde{q}_{ij}\}_{j=1}^{m_i}$ denote the set of points on $\tilde{C}$ lying above $z_i$. The assumption that $(C,\pn)$ is prestable implies that each connected component of $(\tilde{C}, \{ p_{j}\}_{j=1}^{r}, \{\tilde{q}_{1j}\}_{j=1}^{m_1}, \ldots, \{\tilde{q}_{kj}\}_{j=1}^{m_k})$ is stable. Thus, we may define $g$ by sending
\[
\coprod_{i=1}^{k} (Z_i, \{p_{ij}\}_{j=1}^{l_i}, \{q_{ij}\}_{j=1}^{m_i}) \rightarrow (\tilde{C} \cup Z_1 \cup \ldots \cup Z_k, \{ p_{j}\}_{j=1}^{r}, \{p_{ij}\}_{j=1}^{l_1}, \ldots, \{p_{ij} \}_{j=1}^{l_{k}} ),
\]
where $\tilde{C}$ and $\coprod_{i=1}^{k} Z_{i}$ are glued by identifying $\tilde{q}_{ij} \sim q_{ij}$.

If $(C^{s}, \spn)$ is in the image of $g$, then we can write
\[(C^{s},\spn) = (\tilde{C} \cup Z_1 \cup \ldots \cup Z_k, \{ p_{j}\}_{j=1}^{r}, \{p_{ij}\}_{j=1}^{l_1}, \ldots, \{p_{ij} \}_{j=1}^{l_{k}} ),
\]
and we may define a contraction
\[
(C^{s},\spn) \rightarrow (C,\pn)
\]
by collapsing $Z_1, \ldots, Z_k$ to $z_1, \ldots, z_k$, and mapping $\tilde{C}$ birationally onto $C$.

Conversely, we claim that if $\phi:(C^{s}, \spn) \rightarrow (C,\pn)$ is any contraction, then $(C^{s}, \spn)$ is in the image of $g$. First, let us show that $\phi(\Exc(\phi))=\{z_1, \ldots, z_k\}$. Since a stable curve has no points satisfying (1), (2), or (3), it is clear that $\{z_1, \ldots, z_k\} \subset \phi(\Exc(\phi))$. Conversely, if $z \in C$ does not satisfy (1), (2), or (3), then $z$ is either an unmarked node, a marked smooth point, or an unmarked smooth point. In either case, since the genus of a node or a smooth point is zero, $\phi^{-1}(z)$ must be a reduced, connected, arithmetic genus zero curve with only two distinguished points. But this is impossible, since $(C^{s}, \spn)$ is stable.

Now, if $\phi:(C^{s}, \spn) \rightarrow (C,\pn)$ is any contraction, then we have (see Remark \ref{R:Normalize})
\[
C^{s}=\tilde{C} \cup Z_1 \cup \ldots \cup Z_k,
\]
where $\phi(Z_i)=z_i$, and $\tilde{C}$ is the normalization of $C$ at $z_1, \ldots, z_k$. If $z_i$ is a singularity with $m_i$ branches, then $Z_i$ meets $Z_i^{c}$ at $m_i$ points.  If $z_{i}$ supports $l_{i}$ marked points, then $Z_{i}$ supports the same set of marked points. Thus, if we mark points of attachment in the usual way, we may consider $Z_i$ as an $l(z_i)+m(z_i)$-pointed stable curve of genus $g(z_i)$, for each $i=1, \ldots, k$. This shows that $(C^s, \spn)$ is in the image of $g$.
\end{proof}
\begin{corollary}\label{C:Independence}
Let $(C,\pn)$ be a $\Z$-stable curve, and suppose that
\[
\phi:(C^{s}, \spn) \rightarrow (C,\pn)
\]
is any contraction from a stable curve $(C^{s}, \spn)$. Then $\Exc(\phi)=\Z(C^{s})$.
\end{corollary}
\begin{proof}
By Lemma \ref{L:Mapping}, there exists a map $g:\prod_{i=1}^{k} \SM_{g_i, m_i+l_i} \rightarrow \SM_{g,n},$ such that $(C^{s}, \spn)$ admits a contraction to $(C, \pn)$ iff $(C^{s}, \spn) \in \Image(g).$ The pull-back of the universal curve $\C \rightarrow \SM_{g,n}$ via $g$ decomposes as
$$
(\tilde{C} \times \prod_{j=1}^{k}\SM_{g_j,m_j+l_j}  ) \coprod \left(\coprod_{j=1}^{k}\C_i \right),
$$
where $\C_{i}$ is the pull-back of the universal curve over $ \SM_{g_i,m_i+l_i}$ via the $i^{th}$ projection $\prod_{j=1}^{k}\SM_{g_j,m_j+l_j}  \rightarrow \SM_{g_i,m_i+l_i}$. For any geometric point $x \in \prod_{i=1}^{k} \SM_{g_i, m_i+l_i}$, let $C^{s}_x$ denote the fiber $\pi^{-1}(x)$. Then $C^{s}_x$ decomposes as
\[C^{s}_x= \tilde{C} \cup (\C_1)_{x} \cup \ldots \cup (\C_{k})_{x},
\]
and there exists a contraction $\phi: C^{s}_x \rightarrow C$ with $\Exc(\phi)=(\C_1)_{x} \cup \ldots \cup (\C_{k})_{x}$. 

The hypothesis that $(C, \pn)$ is $\Z$-stable implies that \emph{there exists} a geometric point $y \in \prod_{i=1}^{k} \SM_{g_i, m_i+l_i}$ such that $\Z(\C^{s}_{y})=\cup_{i=1}^{k}(\C_{i})_y$. To prove the corollary, we must show $\Z(\C^{s}_{x})=\cup_{i=1}^{k}(\C_{i})_x$ for \emph{every} geometric point $x \in \prod_{i=1}^{k} \SM_{g_i, m_i+l_i}.$ This follows easily from two applications of the one-parameter specialization property for extremal assignments: First, let $\zeta \in  \prod_{i=1}^{k} \SM_{g_i, m_i+l_i}$ be the generic point, and consider a map $\Delta \rightarrow  \prod_{i=1}^{k} \SM_{g_i, m_i+l_i}$ sending $\eta \rightarrow \zeta$, $0 \rightarrow y$. Applying
Definition \ref{D:Assignment} (3) to the induced family over $\Delta$, we conclude 
$$\Z(\C^{s}_{y})=(\C_{1})_y \cup \ldots \cup (\C_{k})_{y} \implies \Z(\C^{s}_{\overline{\zeta}})=(\C_{1})_{\overline{\zeta}} \cup \ldots \cup (\C_{k})_{\overline{\zeta}}.$$
Next, let $x \in  \prod_{i=1}^{k} \SM_{g_i, m_i+l_i}$ be an arbitrary geometric point, and consider a map $\Delta \rightarrow  \prod_{i=1}^{k} \SM_{g_i, m_i+l_i}$ sending $\eta \rightarrow \zeta, 0 \rightarrow x$. Applying
Definition \ref{D:Assignment} (3) to the induced family over $\Delta$, we see that 
$$\Z(\C^{s}_{\overline{\zeta}})=(\C_{1})_{\overline{\zeta}} \cup \ldots \cup (\C_{k})_{\overline{\zeta}} \implies \Z(\C^{s}_{x})=(\C_{1})_{x} \cup \ldots \cup (\C_{k})_{x}.$$

\begin{comment}
 giving rise to a Cartesian diagram
\[
\xymatrix{
(\tilde{C} \times \prod_{j=1}^{k}\SM_{g_j,m_j+l_j}  ) \times \prod_{j=1}^{k}\pi_i^*\C_i \ar[dr] \ar[r]^{\,\,\,\,\,\,\,\,\,\,\,\,\simeq}& \C \times_{\SM_{g,n}} \prod_{j=1}^{k}\SM_{g_j,m_j+l_j}  \ar[r] \ar[d]&\C \ar[d]\\
&\prod_{j=1}^{k}\SM_{g_j,m_j+l_j} \ar[r]&\SM_{g,n},\\
}
\]
such that any stable curve $(C^s,\spn)$ admitting a contraction to $(C,\pn)$ lies in the image of $\Z$.
The hypothesis that $(C,\pn)$ is $\Z$-stable precisely implies that there exists a $k$-point 
\[
\Spec k \rightarrow \prod_{j=1}^{k}\SM_{g_j,m_j+l_j},
\] 
such that $\phi^{*}\L$ has degree zero on every irreducible component of the fibers $(\pi_1^*\C_{1})_{k}, \ldots, (\pi_k^*\C_{k})_{k}$. Since the degree of $\phi^{*}\L$ is constant in flat families, we have that $\phi^{*}\L$ has degree zero on every geometric fiber of  $(\pi_1^*\C_{1})_{k}, \ldots, (\pi_k^*\C_{k})_{k}$. Finally, since $\L$ is $\pi$-nef, $\L$ has non-negative degree on every irreducible component 
such that $\L$ has degree zero on the fibers of $\pi_1^*\C_1, \ldots, \pi_l^*\C_l$. But since $\L$ is relatively nef, it follows immediately that $\L$ must have degree zero on every irreducible component of every fiber of $\pi_i^*\C_i \rightarrow \SM_{g_1,k_i} \times \ldots \SM_{g_l,k_l}$.
\end{comment}
\end{proof}

\begin{comment}
\begin{lemma}\label{L:Independence}
Let $S$ be a normal scheme, of finite-type over $k$, $X \rightarrow S$ a stable curve over $S$ with smooth generic fiber, $Y \rightarrow S$ any curve. Suppose that $\phi:X \rightarrow Y$ is a proper $S$-morphism, satisfying $\phi_*\O_{X} = \O_{Y}.$ Then
$Y_{\overline{s}}$ is $\L$-stable iff $X_{\overline{s}} \rightarrow Y_{\overline{s}}$ is the map associated to $\L_{s}$.
\end{lemma}
\begin{proof}
If $X_{\overline{s}} \rightarrow Y_{\overline{s}}$ is the map associated to $\L_{s}$, then $Y_{\overline{s}}$ is $\L_{s}$ by definition.
The reverse implication, however, is non-trivial. We must show that if $(C,p_1, \ldots, p_n)$ is an $\L$-stable curve, base-change using a dvr.
\end{proof}
\end{comment}

\section{Construction of $\SM_{g,n}(\Z)$}\label{S:Construction}
Throughout this section, we fix an extremal assignment $\Z$ over $\SM_{g,n}$.

\begin{definition}[The moduli stack of $\Z$-stable curves]
Let $\C \rightarrow \V_{g,n}$ be the universal curve over $\V_{g,n}$, the main component in the stack of all curves (Section \ref{S:MainResult}). We define $\SM_{g,n}(\Z) \subset \V_{g,n}$ as the collection of points $\Spec k \rightarrow \SV_{g,n}$ such that the geometric fiber $\C \times_{\V_{g,n}} \overline{k}$ is $\Z$-stable.
\end{definition}

The first main theorem of this paper is
\begin{theorem}\label{T:Construction}
$\SM_{g,n}(\Z) \subset \SV_{g,n}$ is a stable modular compactification of $\M_{g,n}$.
\end{theorem}
In Section \ref{S:Openness}, we will show that $\SM_{g,n}(\Z) \subset \SV_{g,n}$ is Zariski-open. Thus, $\SM_{g,n}(\Z)$ inherits the structure of an algebraic stack, locally of finite-type over $\Spec \mathbb{Z}$. Since a $\Z$-stable curve has no more irreducible components than an $n$-pointed stable curve of arithmetic genus $g$, the moduli problem of $\Z$-stable curves is bounded (see Corollary \ref{C:Stackgne}). Thus, $\SM_{g,n}(\Z)$ inherits the structure of an algebraic stack of finite-type over $\Spec \mathbb{Z}$, and we may use the valuative criterion to check that $\SM_{g,n}(\Z)$ is proper over $\Spec \mathbb{Z}$. This is accomplished in Section \ref{S:Properness}. It follows that $\SM_{g,n}(\Z)$ is a modular compactification of $\M_{g,n}$. To see that $\SM_{g,n}(\Z)$ is a stable modular compactification, simply observe that any $\Z$-stable curve is obviously prestable, since it is obtained as a contraction from a stable curve.

\subsection{$\SM_{g,n}(\Z)$ is open in $\V_{g,n}$}\label{S:Openness}

\begin{lemma}\label{L:Rigidity}
Suppose we have a diagram
\[
\xymatrix{
\C^{s} \ar[rr]^{\phi} \ar[dr]&&\C \ar[ld]\\
&T \ar@/^1pc/[lu]^{\{\sigma^s_i\}_{i=1}^{n}} \ar@/_1pc/[ru]_{\{\sigma_i\}_{i=1}^{n}} &
}
\]
satsifying:
\begin{itemize}
\item[(1)] $T$ is a noetherian scheme.
\item[(2)] $(\C^{s} \rightarrow T, \{\sigma_{i}^{s}\}_{i=1}^{n})$ is a stable curve, and  $(\C \rightarrow T, \sigman)$ an arbitrary curve.
\item[(3)]  $\phi:\C^{s} \rightarrow \C$  is a birational morphism.
\end{itemize}
Then the set
\[
S:=\{t \in T \, |\, \Exc(\phi_{\overline{t}})=\Z(C_{\overline{t}})\}
\]
is open in $T$.
\end{lemma}
\begin{proof}
Since $T$ is noetherian, it suffices to prove that $S$ is constructible and stable under generalization. First, we show that $S$ is constructible. There exists a finite stratification of $T$ into locally-closed subschemes over which the dual graph of the fibers of $\pi$ are constant, i.e. we have
$
T=\coprod_{G}T_{G},$ where $T_{G}:=\M_{G} \times_{\SM_{g,n}} T.$
We will prove that $T$ is a finite union of the subschemes $T_{G}$ by showing that
\[ \Exc(\phi_{\overline{t}})=\Z(C_{\overline{t}}) \text{ for one point $t \in T_{G}$} \implies
\Exc(\phi_{\overline{t}})=\Z(C_{\overline{t}})\text{ for all points $t \in T_{G}$.}
\]
There exists a finite surjective map $\tilde{T}_{G} \rightarrow T_{G}$, such that
$$
\C^{s} \times _{T} \tilde{T}_G=\C_{1}^{s} \cup \ldots \cup \C_{k}^{s},
$$
where each $\C_{i}^{s} \rightarrow \tilde{T}_{G}$ is a smooth proper curve, and it is enough to prove that 
\[ \Exc(\tilde{\phi}_{\overline{t}})=\Z(C_{\overline{t}}) \text{ for one point $t \in \tilde{T}_{G}$} \implies
\Exc(\tilde{\phi}_{\overline{t}})=\Z(C_{\overline{t}})\text{ for all points $t \in \tilde{T}_{G}$,}
\]
where $\tilde{\phi}:\C^{s} \times _{T} \tilde{T}_G \rightarrow \C \times _{T} \tilde{T}_G$ is the birational morphism induced by $\phi$.
Since each $\C_{i}^{s} \rightarrow \tilde{T}_{G}$ is smooth and proper, the rigidity lemma implies that
$$
(\C^{s}_{i})_{\overline{t}} \subset \Exc(\tilde{\phi}_{\overline{t}}) \text{ for one point $t \in \tilde{T}_{G}$} \implies
(\C^{s}_{i})_{\overline{t}} \subset \Exc(\tilde{\phi}_{\overline{t}}) \text{ for all points $t \in \tilde{T}_{G}$.}
$$
On the other hand, since the dual graph of the fibers of $\C^{s} \times_T \tilde{T}_{G} \rightarrow \tilde{T}_G$ is constant, we have
$$
(\C^{s}_{i})_{\overline{t}} \subset \Z(C_{\overline{t}}) \text{ for one point $t \in \tilde{T}_{G}$} \implies
(\C^{s}_{i})_{\overline{t}} \subset \Z(C_{\overline{t}}) \text{ for all points $t \in \tilde{T}_{G}$.}
$$
It follows that
$$ \Exc(\phi_{\overline{t}})=\Z(C_{\overline{t}}) \text{ for one point $t \in \tilde{T}_{G}$} \implies
\Exc(\phi_{\overline{t}})=\Z(C_{\overline{t}})\text{ for all points $t \in \tilde{T}_{G}$,}
$$
as desired.

Next, we show that $S$ is stable under generization. If $s, t \in T$ satisfy $s \in \overline{\{t\}}$, there exists a map $\Delta \rightarrow T$, sending $\eta \rightarrow t$,  $0 \rightarrow s$, inducing a diagram
\[
\xymatrix{
\C^{s} \ar[dr]\ar[rr]^{\phi} & &\C \ar[ld]\\
&\Delta&\\
}
\]
We wish to show that
\[\Exc(C^{s}_{0})=\Z(C^{s}_0) \implies \Exc(C^{s}_{\overline{\eta}})=\Z(C^{s}_{\overline{\eta}}).
\]
As with step 1, this is an elementary application of the rigidity lemma. After a finite base-change, we may assume that the irreducible components of $\C^{s}$, say $\C^{s}=\C^{s}_1 \cup \ldots \cup \C^{s}_{k}$ are in bijective correspondence with the irreducible components of $\C^{s}_{\overline{\eta}}$. By Definition \ref{D:Assignment} (3), we have
\[(\C^{s}_i)_0 \subset \Z(\C^{s}_0) \iff (\C^{s}_i)_{\overline{\eta}} \subset \Z(\C^{s}_{\overline{\eta}})\]
On the other hand, since each $\C_{i} \rightarrow \Delta$ is flat and proper with irreducible generic fiber, the rigidity lemma implies that
\[(\C^{s}_i)_0 \subset \Exc(\phi_0) \iff (\C^{s}_i)_{\overline{\eta}} \subset \Exc(\phi_{\overline{\eta}}).\]
We conclude that $\Exc(\C^{s}_{0})=\Z(\C^{s}_0) \implies \Exc(\C^{s}_{\overline{\eta}})=\Z(\C^{s}_{\overline{\eta}})$ as desired.
\end{proof}

\begin{theorem}
$\SM_{g,n}(\Z) \subset \SV_{g,n}$ is an open substack.
\end{theorem}
\begin{proof}
Since $\SV_{g,n}$ is an algebraic stack, irreducible and locally of finite-type over $\Spec \mathbb{Z}$, there exists a smooth atlas $T \rightarrow \SV_{g,n}$, where $T$ is a scheme, irreducible and locally of finite-type over $\Spec \mathbb{Z}$. If $(\C \rightarrow T, \sigman)$ denotes the corresponding $n$-pointed curve over $T$, then the required statement is that
\[
S:=\{t \in T \, | \, (\C_{\overline{t}}, \{\sigma_i(\overline{t})\}_{i=1}^{n})\text{ is $\Z$-stable} \}
\]
is open in $T$. Since this is local on $T$, we may assume that $T$ is irreducible and of finite-type over $\Spec \mathbb{Z}$.

Note that if $p: \tilde{T} \rightarrow T$ is any proper surjective morphism of schemes, and $(\tilde{\C} \rightarrow \tilde{T}, \{\tilde{\sigma}_i\}_{i=1}^{n})$ is the family obtained by pull-back, then it is sufficient to show that 
\[
\tilde{S}:=\{t \in \tilde{T} \, | \, (\tilde{\C}_{\overline{t}}, \{\tilde{\sigma}_i(\overline{t})\}_{i=1}^{n})\text{ is $\Z$-stable} \}
\]
is open in $\tilde{T}$. By Lemma \ref{L:ExtendingCurves}, there exists an alteration $\tilde{T} \rightarrow T$, and a diagram
\[
\xymatrix{
\C^{s} \ar@{-->}^{\phi}[rr] \ar[dr] && \tilde{\C} \ar[dl]\\
&\tilde{T}  \ar@/^1pc/[lu]^{\{\sigma^s_i\}_{i=1}^{n}} \ar@/_1pc/[ru]_{\{\tilde{\sigma}_i\}_{i=1}^{n}}&
}
\]
satisfying
\begin{enumerate}
\item $\tilde{T}$ is a normal noetherian scheme.
\item $(\C^{s}, \{\sigma^s_i\}_{i=1}^{n})$ is a stable curve.
\item $\phi$ is regular over the locus $\{t \in \tilde{T} |\text{ $\tilde{C}_{t}, \tilde{\sigma}(t)$ is prestable }\}$.
\end{enumerate}

In particular, since every $\Z$-stable curve is prestable, $\tilde{S}$ is contained in the open set 
$$U:=\{ t \in \tilde{T} | \text{ $\phi$ is regular in a neighborhood of the fiber $\C^{s}_t$ }\}.$$

Thus, we may replace $\tilde{T}$ by $U$ and assume that $\phi$ is regular. By Lemma \ref{L:BirationalBaseChange}, the restriction of $\phi$ to each fiber is a contraction of curves. Thus, by Corollary \ref{C:Independence},
$$
\tilde{S}=\{t \in \tilde{T} \,\,|\,\, \Exc(\phi_{t})=\Z(C^{s}_{t}) \}.
$$
By Lemma \ref{L:Rigidity}, this set is open.
\end{proof}
\subsection{$\SM_{g,n}(\Z)$ is proper over $\Spec \mathbb{Z}$}\label{S:Properness}
To show that $\SM_{g,n}(\Z)$ is proper, it suffices to verify the valuative criterion for discrete valuation rings with algebraically closed residue field, whose generic point maps into the open dense substack $\M_{g,n} \subset \SM_{g,n}(\Z)$ (\cite{LMB}, Chapter 7).

\begin{theorem}[Valuative Criterion for Properness of $\SM_{g,n}(\Z)$]
Let $\Delta$ be the spectrum of a discrete valuation ring with algebraically closed residue field.
\begin{itemize}
\item[(1)] (Existence of $\Z$-stable limits) If $(\C,\sigman )|_{\eta}$ is a smooth $n$-pointed curve over $\eta$, there exists a finite base-change $\Delta' \rightarrow \Delta$, and a $\Z$-stable curve $(\C' \rightarrow \Delta', \sigmanprime)$, such that
$$(\C', \sigmanprime )|_{\eta'} \simeq (\C,\sigman)|_{\eta} \times_{\eta} \eta'.$$
\item[(2)] (Uniqueness of $\Z$-stable limits) Suppose that $(\C \rightarrow \Delta,\sigman)$  and $(\C' \rightarrow \Delta, \sigmanprime)$ are $\Z$-stable curves with smooth generic fiber. Then any isomorphism over the generic fiber
$$(\C,\sigman)|_{\eta} \simeq (\C',\sigmanprime)|_{\eta}$$ extends to an isomorphism over $\Delta$:
$$(\C,\sigman) \simeq (\C',\sigmanprime).$$
\end{itemize}
\end{theorem}
\begin{proof}
To prove existence of limits, start by applying the stable reduction theorem to $(\C,\sigman )|_{\eta}$.  There exists a finite base-change $\Delta' \rightarrow \Delta$, and a Deligne-Mumford stable curve $(\pi: \C^{s} \rightarrow \Delta',\sigmanprime)$ such that
$$
(\C^{s}, \sigmans)|_{\eta'} \simeq (\C,\sigman) \times_{\eta} \eta'.
$$
For notational simplicity, we will continue to denote our base by $\Delta$. By Proposition \ref{P:Contractions}, there exists a birational morphism over $\Delta$:
\[
\xymatrix{
\C^{s} \ar[rr]^{\phi} \ar[dr]&&\C \ar[dl]\\
&\Delta&
}
\]
such that
\begin{enumerate}
\item $(\C \rightarrow \Delta, \sigman)$ is a flat family of $n$-pointed curves,
\item $\phi$ is proper birational with $\Exc(\phi)=\Z(C^{s}_0)$,
\item $\phi_{0}:C^{s}_0 \rightarrow C_0$ is a contraction of curves.
\end{enumerate}
Properties (2) and (3) imply that the special fiber $(C_0, \{\sigma_i(0)\}_{i=1}^{n})$ is $\Z$-stable, so $(\C \rightarrow \Delta, \sigman)$ is the desired $\Z$-stable family.\\

To prove uniqueness of limits, we must show that a rational isomorphism
$$
(\C,\sigman)|_{\eta} \simeq (\C',\sigmanprime)|_{\eta}
$$
between two families of $\Z$-stable curves extends to an isomorphism over $\Delta$. It suffices to check that the rational map $\C \dashrightarrow \C'$ extends to an isomorphism after a finite base-change. Thus, applying semistable reduction to the graph of this rational map, we may assume there exists a nodal curve $(\C^{n} \rightarrow \Delta, \taun)$ and a diagram
\[
\xymatrix{
&(\C^{n}, \taun) \ar[rd]^{\phi'} \ar[ld]_{\phi} &\\
(\C,\sigman)& &(\C', \sigmanprime)
}
\]
where $\phi$ and $\phi'$ are proper birational morphisms over $\Delta$. In fact, we may further assume that $(\C^{n} \rightarrow \Delta, \taun)$ is stable. Indeed, any rational component $E \subset C^{n}_{0}$ with only one or two distinguished points must be contracted by both $\phi$ and $\phi'$ since $C_0$ and $C_0'$ are both prestable. Thus, $\phi$ and $\phi'$ both factor through the stable reduction $\C^{n} \rightarrow \C^{s}$, and we may replace by $\C^{n}$ by $\C^{s}$.

Now consider the restriction of $\phi$ and $\phi'$ to the special fiber. By Lemma \ref{L:BirationalBaseChange}, $\phi_{0}:C^{s}_0 \rightarrow C_0$ and $\phi_{0}':C^{s}_0 \rightarrow C_0'$ are both contractions of curves. Furthermore, since $C_{0}$ and $C_{0}'$ are both $\Z$-stable, Corollary \ref{C:Independence} implies that $\Exc(\phi)=\Exc(\phi')=\Z(C^{s}_0)$. Since $\C$ and $\C'$ are normal (Lemma \ref{L:Normality}), this implies $\C \simeq \C'$ as desired.
\end{proof}

\section{Classification of stable modular compactifications of $\M_{g,n}$}\label{S:Classification}

In this section, we prove the following theorem.

\begin{theorem}[Classification of Stable Modular Compactifications]\label{T:Classification}
Suppose $\X \subset \V_{g,n}$ is a stable modular compactification of $\M_{g,n}$. Then there exists an extremal assignment $\Z$ over $\SM_{g,n}$, such that $\X=\SM_{g,n}(\Z)$.
\end{theorem}

The starting point for the proof of Theorem \ref{T:Classification} is the following lemma, which allows us to compare an arbitrary stable modular compactification to $\SM_{g,n}$ by regularizing the rational map between their respective universal curves.

\begin{lemma}\label{L:Diagram}
Suppose that $\X \subset \V_{g,n}$ is a stable modular compactification of $\M_{g,n}$. Then there exists a diagram
\[
\xymatrix{
\C^{s} \ar[dr]^{\pi^{s}} \ar[rr]^{\phi}&&\C \ar[dl]_{\pi}\\
&T \ar[dl]_{p} \ar[dr]^{q} \ar@/_1pc/[ru]_{\sigman}  \ar@/^1pc/[lu]^{\{\sigma_i^{s} \}_{i=1}^{n}}&\\
\overline{M} \ar[d]_{i}&U\ar[d] \ar@{^{(}->}[r] \ar@{_{(}->}[l]&X \ar[d]^{j}\\
\SM_{g,n}&\U \ar@{^{(}->}[r] \ar@{_{(}->}[l]&\X\\
}
\]
satisfying
\begin{itemize}
\item[(1)] $X$, $U$, $\overline{M}$, and $T$ are irreducible normal schemes, of finite-type over $\Spec \mathbb{Z}$.
\item[(2)] $p$ and $q$ are proper birational morphisms, $i$ and $j$ are generically-\'{e}tale finite morphisms, $\pi^{s}:\C^{s} \rightarrow T$ and $\pi: \C \rightarrow T$ are the families induced by $i \circ p$ and $j \circ q$ respectively.
\item [(3)] $\U:=\M_{g,n} \cap \X$ is an open dense substack of $\X$ and $\M_{g,n}$. The lower squares are Cartesian, and the unlabelled arrows are open immersions.
\item[(4)] The rational map $\phi$, induced by the natural isomorphism between $\C^{s}$ and $\C$ over the generic point of $T$, is regular.
\end{itemize}
\end{lemma}
\begin{proof}
Since $\X \cup \SM_{g,n}$ is a (non-separated) algebraic stack of finite-type over $\Spec \mathbb{Z}$ with quasi-finite diagonal, there exists a representable, finite, generically-\'{e}tale, surjective morphism $S \rightarrow \X \cup \SM_{g,n}$, where $S$ is a scheme (\cite{EHKV}, Theorem 2.7). We may assume that $S$ is irreducible since $\X \cup \SM_{g,n}$ is. Define $X$, $\overline{M}$, and $U$ as the normalizations of the fiber products $S \times_{\X \cup \SM_{g,n}} \X$, $S \times_{\X \cup \SM_{g,n}} \SM_{g,n}$, and $S \times_{\X \cup \SM_{g,n}} \U$ respectively.  Finally, define $T$ to be the normalization of the closure of the image of the diagonal immersion $U \hookrightarrow X \times_{\Spec \mathbb{Z}} \overline{M}$. This gives a diagram satisfying (1), (2), and (3), but not necessarily (4), i.e. the induced rational map $\phi: \C^{s} \dashrightarrow \C$ may not be regular. 

Since the geometric fibers of $(\C \rightarrow T, \sigman)$ are prestable, Lemma \ref{L:ExtendingCurves} gives an alteration $T' \rightarrow T$ such that the rational map $\C^{s} \times_{T} T' \dashrightarrow \C \times_{T} T'$ is regular. Now define $X' \rightarrow X$ and $\overline{M}' \rightarrow \overline{M}$ to be the finite morphisms appearing in the Stein factorizations of $T' \rightarrow X$ and $T' \rightarrow \overline{M}$ respectively. Replacing $X$, $\overline{M}$, $T$ by $X'$, $\overline{M}'$, $T'$, and $U$ by $U \times_{\overline{M}} T'=U \times_{X} T'$ gives the desired diagram.
\end{proof}

For the remainder of this section, we fix a stable modular compactification $\X$, and a diagram as in Lemma \ref{L:Diagram}. We also use the following notation: For any geometric point $t \in T$, let $G_{t}$ be the dual graph of the fiber $(C^{s}_t, \{\sigma_i^{s}(t)\}_{i=1}^{n})$, so that $\Exc(\phi_t)$ determines a subgraph $\Exc(\phi_t) \subset G_t$. Also, we set $T_{G}:= \M_{G} \times_{\SM_{g,n}} T$, i.e. $T_{G} \subset T$ is the locally closed subscheme in $T$ over which the geometric fibers of $\pi^{s}$ have dual graph isomorphic to $G$.

 We wish to associate to $\X$ an extremal assignment $\Z$ by setting
\[
\Z(G):=i(\Exc(\phi_t)) \subset G,
\]
for some choice of $t \in T_{G}$ and some choice of isomorphism $i:G_{t} \simeq G$. The key point is to show that the subgraph $\Z(G) \subset G$ does not depend on these choices. More precisely, we will show:

\begin{proposition}\label{P:ExceptionalLocus}
\begin{itemize}
\item[]
\item[(a)] For any two geometric points $t_1, t_2 \in T_{G}$, there exists an isomorphism $i:G_{t_1} \simeq G_{t_2}$ such that $i(\Exc(\phi_{t_1}))=\Exc(\phi_{t_2}).$
\item[(b)] For any geometric point $t \in T_{G}$, and any automorphism $i:G_{t} \simeq G_{t}$, we have $i(Exc(\phi_t))=\Exc(\phi_t)$, i.e. $\Exc(\phi_t)$ is $\Aut(G_t)$-invariant.
\end{itemize}
\end{proposition}

Before proving this proposition, let us use it to prove Theorem \ref{T:Classification}.

\begin{proof}[Proof of Theorem \ref{T:Classification}, assuming Proposition \ref{P:ExceptionalLocus}.]
If $G$ is any dual graph of an $n$-pointed stable curve of genus $g$, we define a subgraph $\Z(G) \subset G$ by the recipe
$$
\Z(G):=i(\Exc(\phi_t)) \subset G,
$$
for any choice of $t \in T_{G}$ and isomorphism $i:G_{t} \simeq G$. By Proposition \ref{P:ExceptionalLocus}, the subgraph of $\Z(G) \subset G$ does not depend on the choice of $t \in T_{G}$ or the choice of isomorphism $i:G_{t} \simeq G$. We claim that the assignment $G \rightarrow \Z(G)$ defines an extremal assignment over $\SM_{g,n}$, and that $\X=\SM_{g,n}(\Z)$.

To prove that $\Z$ is an extremal assignment, we must show that it satisfies axioms 1-3 of Definition \ref{D:Assignment}. For axiom 1, suppose that $\Z(G) =G$ for some dual graph $G$. Then there exists a point $t \in T$ such that $\Exc(\phi_t)=C^{s}_t$. Since $T$ is connected, and the locus
$$\{ t \in T |\,\, \Exc(\phi_t)=C^s_{t} \}$$
is both open and closed in $T$, it must follow that $\Exc(\phi_t)=C^{s}_t$ \emph{for every} $t \in T$. But this is impossible since $\phi$ is an isomorphism over the generic point of $T$. We conclude that $\Z(G) \neq G$ for every dual graph $G$.

Axiom 2 is immediate from Proposition \ref{P:ExceptionalLocus} (b).

Finally, for axiom 3, suppose that $G \leadsto G'$ is an arbitrary specialization of dual graphs. Let $x \in \SM_{g,n}$ be the generic point of $\M_{G}$, let $y \in \SM_{g,n}$ be a closed point of $\M_{G'}$, and let $t \in T$ be any point lying over $x$. Since $y \in \overline{\{x\}}$, there exists a diagram
\[
\xymatrix{
\eta \ar[d] \ar[r]^{\tilde{u}}&T \ar[d]\\
\Delta \ar[r]^{u}& \SM_{g,n}
}
\]
satisfying $u(\eta)=x$, $u(0)=y$, $\tilde{u}(\eta)=t$. Since $T \rightarrow \SM_{g,n}$ is proper, this diagram fills in (possibly after a finite base-change) to give map $\tilde{u}:\Delta \rightarrow T$, and we may consider the induced morphism of families pulled back from $T$:
\[
\xymatrix{
\C^{s}  \ar[dr] \ar[rr]^{\phi} && \C \ar[dl]\\
&\Delta&\\
}
\]
Since our definition of $\Z$ does not depend on the choice of $t \in T$, we have $\Z(G)=\Exc(\phi_{\overline{\eta}})$ and $\Z(G')=\Exc(\phi_{0})$. After a finite base-change, we may assume that $\C^{s}=\C_{1} \cup \ldots \cup \C_{m}$, where each $\C_{i} \rightarrow \Delta$ is a flat family of curves with smooth generic fiber, and it suffices to show that $$(\C_{i})_{\bar{\eta}} \in \Exc(\phi_{\overline{\eta}}) \iff (\C_{i})_{0} \in \Exc(\phi_0).$$
But this is an immediate consequence of the rigidity lemma. This completes the proof that $\Z$ is an extremal assignment over $\SM_{g,n}$.

Since $\Z$ is an extremal assignment, Theorem \ref{T:Construction} gives a stable modular compactification $\SM_{g,n}(\Z) \subset \V_{g,n}$, and we wish to show that $\X=\SM_{g,n}(\Z).$
By Lemma \ref{L:BirationalBaseChange}, the map $\phi_t:C^{s}_t \rightarrow C_t$ is a contraction for any geometric point $t \in T$. Since $\Exc(\phi_t)=\Z(C^{s}_t)$ by the definition of $\Z$, every geometric fiber of $(\C \rightarrow T, \{\sigma_i\}_{i=1}^{n})$ is $\Z$-stable. Since $T \rightarrow \X$ is surjective, we conclude that every geometric point of $\X$ is contained in $\SM_{g,n}(\Z)$. Since $\SM_{g,n}(\Z)$ and $\X$ are open in $\V_{g,n}$, the natural inclusion $\X \subset \SM_{g,n}(\Z)$ is an open immersion. Furthermore, this inclusion is proper over $\Spec \mathbb{Z}$, since both $\X$ and $\SM_{g,n}(\Z)$ are. A proper dominant morphism is surjective, so $\X = \SM_{g,n}(\Z)$ as desired.

\end{proof}

It remains to prove Proposition \ref{P:ExceptionalLocus}. The major difficulty comes from the fact that the subscheme $T_{G} \subset T$ may have several connected components. The following lemma, used in conjunction with the valuative criterion for $\X$, will allow us to relate the exceptional loci of the contractions $\phi_{t}$, even as $t$ ranges over different connected components of $T_{G}$.

\begin{lemma}\label{L:LiftableMap}
Given $T \rightarrow \SM_{g,n}$ as in Lemma \ref{L:Diagram}, let $x \in \SM_{g,n}$ be any geometric point, and let $T_1, \ldots T_k$ be the connected components of the fiber $T_{x}$. Then there exists a diagram

\[
\xymatrix{
&&T \ar[d]\\
\Delta \ar[rr]^{\,\,\,\,u} \ar[urr]^{\{\tilde{u}_i\}_{i=1}^{k}}&& \SM_{g,n}
}
\]
such that
\begin{itemize}
\item[1.] $u(\eta)=\eta_{\SM_{g,n}}$, $u(0)=x$,
\item[2.] $\tilde{u}_i(\eta)=\eta_T$, $\tilde{u}_i(0) \in T_i$ for each $i=1, \ldots, k$,
\end{itemize}
where $\eta_{\SM_{g,n}}$ and $\eta_T$
 are the generic points of $\SM_{g,n}$ and $T$ respectively.
\end{lemma}
\begin{proof}
Let $\{x_{1}, \ldots, x_{k}\}$ be the fiber of $\overline{M} \rightarrow \SM_{g,n}$ over $x$. By Zariski's main theorem, the geometric fibers of $T \rightarrow \overline{M}$ are connected, so the connected components $T_{1}, \ldots, T_{k}$ are simply the fibers of $T \rightarrow \overline{M}$ over $x_1, \ldots, x_k$.

Let $u:\Delta \rightarrow \SM_{g,n}$ be any map sending $\eta \rightarrow \eta_{\SM_{g,n}}$ and $0 \rightarrow x$. Let $S$ be any irreducible component of the fiber product $\overline{M} \times_{\SM_{g,n}} \Delta$ which dominates both $\overline{M}$ and $\Delta$. Since $S \rightarrow \Delta$ is generically etale, we may assume, after a finite base-change, that the generic fiber is contained in a union of sections $\{\sigma_i\}_{i=1}^{d}$. Furthermore, since the generic fiber of $S \rightarrow \Delta$ is dense in $S$, each point of the special fiber $S_{0}=\{x_1, \ldots, x_k\}$ is equal to $\sigma_i(0)$ for some section $\sigma_i$. Reordering if necessary, we may assume that $\sigma_i(0)=x_i$ for $i=1, \ldots, k$.

The sections $\{\sigma_i\}_{i=1}^{k}$ induces lifts of $u$ to $\overline{M}$, i.e we have a diagram

\[
\xymatrix{
&&\overline{M} \ar[d]\\
\Delta \ar[rr]^{\,\,\,\,u} \ar[urr]^{\{u_i\}_{i=1}^{k}}&& \SM_{g,n}
}
\]
where $u_i(\eta)=\eta_{\overline{M}}$ and $u_i(0)=x_i$ for $i=1, \ldots, k$. Since $T \rightarrow \overline{M}$ is proper and an isomorphism over the generic point, each map $u_{i}$ lifts uniquely to a map $\tilde{u}_i:\Delta \rightarrow T$ satisfying $\tilde{u}_i(\eta)=\eta_{T}$ and $\tilde{u}_i(0) \in T_i.$ 
\end{proof}

\begin{itemize}
\item[]
\end{itemize}
\begin{proof}[Proof of Proposition \ref{P:ExceptionalLocus}(a)]
\begin{itemize}
\item[]
\end{itemize}
We will prove the statement in two steps. First, we show that if a pair of points $t_1,t_2$ is contained in a single connected component of $T_{G}$, then there exists an isomorphism $i:G_{t_1} \simeq G_{t_2}$ such that $i(\Exc(\phi_{t_1}))=\Exc(\phi_{t_2})$. Second, we will show that if a pair of points $t_1, t_2 \in T$ is contained in the fiber $T_{x}$ over a geometric point $x \in \SM_{g,n}$, then there exists an isomorphism $i:G_{t_1} \simeq G_{t_2}$ such that $i(\Exc(\phi_{t_1}))=\Exc(\phi_{t_2})$.

It is easy to see that these two statements suffice: Indeed, given two arbitrary geometric points $t_1, t_2 \in T_{G}$ mapping to $x_1, x_2 \in \SM_{g,n}$, let $T_{G}'$ be an irreducible component of $T_{G}$ dominating $\M_{G}$, and let $u_1, u_2 \in T_{G}'$ be two geometric points lying above $x_1$ and $x_2$ respectively. By the second claim, there exist isomorphisms
\begin{align*}
i_1:G_{t_1} \simeq G_{u_1}& \text{ satisfying } i_1(\Exc(\phi_{t_1}))=\Exc(\phi_{u_1}),\\
i_2:G_{t_2} \simeq G_{u_2}& \text{ satisfying } i_2(\Exc(\phi_{t_2}))=\Exc(\phi_{u_2}).\\
\intertext{By the first claim, there exists an isomorphism}
i:G_{u_1} \simeq G_{u_2}& \text{ satisfying } i(\Exc(\phi_{u_1}))=\Exc(\phi_{u_2}).
\end{align*}
Thus, $j:=i_2^{-1} \circ i \circ i_1: G_{t_1} \simeq G_{t_2}$ satisfies $j(\Exc(\phi_{t_1}))=\Exc(\phi_{t_2})$, as desired.\\

To prove the first claim, let $S$ be any connected component of $T_{G}$ and consider the induced diagram
\[
\xymatrix{
\C^{s} \ar[dr] \ar[rr]^{\phi}&& \C \ar[ld]\\
&S&
}
\]
After a finite surjective base-change, we may assume that
$$
\C^{s} \simeq \C^s_{1} \cup \ldots \cup \C^s_{m},
$$
where each $\C^s_{i} \rightarrow S$ is proper and smooth. Note that this isomorphism induces an obvious identification of dual graphs $i:G_{s_1} \simeq G_{s_2}$ for any two geometric points $s_1, s_2 \in S$. Furthermore, since one geometric fiber of $\C^{s}_{i} \rightarrow S$ is contracted by $\phi$ if and only if \emph{every} geometric fiber $\C_{i} \rightarrow S$ is contracted by $\phi$, we have $i(\Exc(\phi_{s_1}))=\Exc(\phi_{s_2})$.\\

To prove the second claim, let $x \in \SM_{g,n}$ be any geometric point and let $T_1, \ldots, T_k$ be the connected components of $T_{x}$. Given the first claim, it suffices to prove that there exist points $t_{i} \in T_{i}$ for each $i=1, \ldots, k$, and isomorphisms
$$
G_{t_1} \simeq G_{t_2} \simeq \cdots \simeq G_{t_k},
$$
identifying $\Exc(\phi_{t_1}) \simeq \Exc(\phi_{t_2})  \simeq \cdots \simeq \Exc(\phi_{t_k})$.

By Lemma \ref{L:LiftableMap}, there exists a commutative diagram
\[
\xymatrix{
&&T \ar[d]\\
\Delta \ar[rr]^{\,\,\,\,u} \ar[urr]^{\{\tilde{u}_i\}_{i=1}^{k}}&& \SM_{g,n}
}
\]
 satisfying
\begin{itemize}
\item[1.] $u(\eta)=\eta_{\SM_{g,n}}$, $u(0)=x$.
\item[2.] $\tilde{u}_i(\eta)=\eta_T$, $\tilde{u}_i(0) \in T_i$ for each $i=1, \ldots, k$.
\end{itemize}
Set $t_{i}:=\tilde{u}_{i}(0)$ for each $i=1, \ldots, k$, and let
\[
\xymatrix{
\C_{i}^{s} \ar[rr]^{\phi_i} \ar[dr]&&\C_{i} \ar[dl]\\
&\Delta&
}
\]
be the diagram induced by the morphism $\tilde{u}_i:\Delta \rightarrow T$. Note that, since the compositions $\Delta \rightarrow T \rightarrow \SM_{g,n}$ are identical when restricted to the generic point $\eta \in \Delta$, we obtain a commutative diagram of isomorphisms over $\eta$:
\[
\xymatrix{
(\C^{s}_1)_{\eta} \ar[r]^{\simeq}  \ar[d]^{(\phi_1)_{\eta}} & (\C^{s}_2)_{\eta} \ar[r]^{\simeq} \ar[d]^{(\phi_2)_{\eta}} & \cdots \ar[r]^{\simeq} &(\C^{s}_k)_{\eta} \ar[d]^{(\phi_k)_{\eta}} \\
(\C_1)_{\eta} \ar[r]^{\simeq} & (\C_2)_{\eta} \ar[r]^{\simeq} & \cdots \ar[r]^{\simeq}&(\C_k)_{\eta}\\
}
\]
Since $\SM_{g,n}$ is proper over $\Spec \mathbb{Z}$, each isomorphism $(\C^{s}_{i})_{\eta} \simeq (\C^{s}_j)_{\eta}$ extends uniquely to an isomorphism $\C^{s}_{i} \simeq \C^{s}_{j}$. Similarly, since $\X$ is proper over $\Spec \mathbb{Z}$, each isomorphism $(\C_{i})_{\eta} \simeq (\C_j)_{\eta}$ extends uniquely to an isomorphism $\C_{i} \simeq \C_{j}$. Thus, we obtain a commutative diagram over $\Delta$:
\[
\xymatrix{
\C^{s}_1 \ar[r]^{\simeq}  \ar[d]^{\phi_1} & \C^{s}_2 \ar[r]^{\simeq} \ar[d]^{\phi_2} & \cdots \ar[r]^{\simeq} &\C^{s}_k \ar[d]^{\phi_k} \\
\C_1 \ar[r]^{\simeq} & \C_2 \ar[r]^{\simeq} & \cdots \ar[r]^{\simeq}& \C_k\\
}
\]
Restricting to the top row of isomorphisms to the special fiber, we obtain isomorphisms
$$
C^{s}_{t_1} \simeq C^{s}_{t_2} \simeq \ldots \simeq C^{s}_{t_k},
$$
identifying $\Exc(\phi_{t_1}) \simeq \ldots \simeq \Exc(\phi_{t_k})$, as desired.\\
\end{proof}

\begin{comment}
\begin{figure}
\scalebox{.50}{\includegraphics{NonseparatedLimits.pdf}}
\caption{Exceptional dual graphs when $(g,n)=(2,0)$ and $(2,1)$.}\label{F:NonseparatedLimits}
\end{figure}
\end{comment}

It remains to prove Proposition \ref{P:ExceptionalLocus} (b). From Proposition \ref{P:ExceptionalLocus} (a), it already follows that every curve parametrized by $\X$ is obtained by contracting some subcurve of a stable curve $Z(C^{s}) \subset C^{s}$, and that this subcurve depends only on the dual graph of $C^{s}$. We will show that the separatedness of $\X$ forces the subcurve $Z(C^{s})$ to be invariant under automorphisms of the dual graph. The idea is simple: if $i$ is an automorphism of the dual graph of $C^{s}$ such that $Z(C^{s}) \neq i(Z(C^{s}))$, then contracting $Z(C^{s})$ or $i(Z(C^{s}))$ in a one-parameter smoothing of $C^{s}$ gives two distinct limits in $\X$. To make this precise, we need two preliminary lemmas.

\begin{lemma}\label{L:LastLemma}
Suppose that $(C,\pn)$ admits a contraction from a stable curve $f:(C^s, \spn) \rightarrow (C,\pn)$. Suppose that there exists $t \in T$ and an isomorphism $i:(C^{s}_t, \{\sigma_{i}(t)\}_{i=1}^{n}) \simeq (C^s, \spn)$ satisfying $i(\Exc(\phi_t))=\Exc(f)$. Then the same holds true for \emph{any} contraction from a stable curve to $(C,\pn)$, i.e given \emph{any} contraction $g:(D^{s}, \spn) \rightarrow (C,\pn)$ with $(D^{s}, \spn)$ stable, there exists  $t \in T$ and an isomorphism $i:(C^{s}_t, \{\sigma_{i}(t)\}_{i=1}^{n}) \simeq (D^s, \spn)$ satisfying $i(\Exc(\phi_t))=\Exc(g)$.
\end{lemma}
\begin{proof}
By Lemma \ref{L:Mapping}, there exists a map
\[
h:\prod_{i=1}^{k} \SM_{g_i, m_i+l_i} \rightarrow \SM_{g,n}
\]
with the property that a stable curve admits a contraction to $(C,\pn)$ iff it lies in the image of $h$. Set $\SM_{C}:=\prod_{i=1}^{k} \SM_{g_i, m_i+l_i}$, let $T_{C}$ be an irreducible component of the fiber product $T \times_{\SM_{g,n}} \SM_{C}$ which dominates $\SM_{C}$, and consider the induced birational morphism of families over $T_{C}$:
\[
\xymatrix{
\C^{s} \ar[rr]^{\phi} \ar[dr]&&\C \ar[dl]\\
&T_{C}&\\
}
\]
Since the family $\C^{s} \rightarrow T_{C}$ is pulled back from $\SM_{C}$, it decomposes as
$$
\C^{s} \simeq (\tilde{C} \times T_{C}) \coprod \left( \coprod_{j=1}^{k}\C^s_i \right),
$$
where $\C^{s}_i \rightarrow T_{C}$ is pulled back from the universal curve over  $\SM_{g_i, m_i+l_i}$ via the $i^{th}$ projection.

Our hypothesis implies that there exists a point $t \in T_{C}$ such that $\Exc(\phi_{t})$ is precisely the union of fibers $\cup_{i=1}^{k}(\C_i)_{t}$. By the rigidity lemma, $\Exc(\phi_t)=\cup_{i=1}^{k}(\C^s_i)_{t}$ for \emph{all} $t \in T_{C}$. Since $T_{C} \rightarrow \SM_{C}$ is surjective, this implies that if $\psi:(D^{s}, \spn) \rightarrow (C,\pn)$ is any contraction from a stable curve to $(C,\pn)$, there exists $t \in T_{C} \subset T$ and an isomorphism $i:C^{s}_t \simeq D^{s}$ such that $i(\Exc(\phi_t))=\Exc(g)$.

\end{proof}

\begin{lemma}\label{L:TwistingX}
Let  $f: (C^s, \spn) \rightarrow (D, \pn)$ be a contraction from a stable curve $(C^s, \spn)$ to a smoothable curve $(D, \pn)$. Suppose that there exists $t \in T$ and an isomorphism $i:(C^{s}_t, \{\sigma_{i}(t)\}_{i=1}^{n}) \simeq (C^{s}, \spn)$ satisfying $i(\Exc(\phi_t))=\Exc(f)$. Then $[D, \pn] \in \X$.
\end{lemma}
\begin{proof}
Since $(D, \pn)$ is smoothable, we may consider a generic smoothing $(\D \rightarrow \Delta, \sigman)$. Since the fibers of $\D \rightarrow \Delta$ are prestable, we may apply Lemma \ref{L:ExtendingCurves} (2) to obtain (after finite base-change) a birational morphism:
\[
\xymatrix{
\D^{s} \ar[rr]^{\psi} \ar[dr]&&\D \ar[dl]\\
&\Delta&\\
}
\]
where $(\D^{s} \rightarrow \Delta, \sigmans)$ is a stable curve. By Lemma \ref{L:BirationalBaseChange}, the restriction of $\psi$ to the special fiber is a contraction morphism $\psi_{0}:(D^s, \spn) \rightarrow (D, \pn)$, though not necessarily the one given in the hypothesis of the lemma. Nevertheless, by Lemma \ref{L:LastLemma}, there exists a geometric point $t \in T$ and an isomorphism $i:(C^{s}_t, \{\sigma_{i}(t)\}_{i=1}^{n}) \simeq (D^{s}, \spn)$ satisfying $i(\Exc(\phi_t))=\Exc(\psi_0)$.

The stable curve $(\D^{s} \rightarrow \Delta, \sigmans)$ induces a map $\Delta \rightarrow \SM_{g,n}$, and we let $T_{\Delta}$ denote an irreducible component of $T \times_{\SM_{g,n}}\Delta$ which dominates $\Delta$. By Proposition \ref{P:ExceptionalLocus} (a), we may assume that the given point $t \in T$ lies in $T_{\Delta}$, so the isomorphism $i:C^{s}_t \simeq D^{s}$ determines a lift of the closed point $\Spec k \hookrightarrow \Delta$ to $T_{\Delta}$, i.e. we have a diagram
\[
\xymatrix{
&T_{\Delta} \ar[d] \ar[r]&T \ar[d]&\\
\Spec k \ar[r] \ar[ur]& \Delta \ar[r]& \SM_{g,n} \\
}
\]
Since the projection $T_{\Delta} \rightarrow \Delta$ is generically-\'{e}tale, we may assume (after finite base-change) that there exists a section $\Delta \rightarrow T_{\Delta}$ whose image contains the given $k$-point. This section induces a lifting $\Delta \rightarrow T$, and we may consider the birational morphism of families pulled back from $T$:
\[
\xymatrix{
\C^{s} \ar[rr]^{\phi} \ar[dr]&&\C \ar[dl]\\
&\Delta&\\
}
\]
By construction, the natural isomorphism $\C^{s} \simeq \D^s$ identifies $\Exc(\phi)$ with $\Exc(\psi)$. Since $\C$ and $\D$ are normal, there is an induced isomorphism $\C \simeq \D$. Thus, the isomorphism class of the curve $(D,\pn)$ appears as a fiber of the family $\C \rightarrow T$, i.e. $[D, \pn] \in \X$ as desired.
\end{proof}

Now we can prove Proposition \ref{P:ExceptionalLocus} (b).
\begin{proof}[Proof of Proposition \ref{P:ExceptionalLocus}(b)]
Suppose there exists a dual graph $G$ and a geometric point $t \in T_{G}$ such that $\Exc(\phi_t)$ fails to be $\Aut(G)$-invariant. By Proposition \ref{P:ExceptionalLocus}(a), \emph{every} geometric point $t \in T_{G}$ has the property that $\Exc(\phi_t)$ fails to be $\Aut(G)$-invariant. In particular, since $T_{G} \rightarrow \M_{G}$ is surjective, we may choose a geometric point $t \in T_{G}$ with the property that the induced map
$$\Aut(C^s_t, \{\sigma_{i}(t)\}_{i=1}^{n}) \rightarrow \Aut(G)$$ is surjective. (Simply choose a curve $[C^{s}, \{p_i\}_{i=1}^{n}] \in \M_{G}$ with the property that each of its components have identical moduli, and take $t \in T_{G}$ to be a point lying over $[C^{s}, \{p_i\}_{i=1}^{n}]$.) Now we will derive a contradiction to the separatedness of $\X$. 

Let $(\C^{s} \rightarrow \Delta, \sigman)$ be a smoothing of the curve $(C^s_t, \{\sigma_{i}(t)\}_{i=1}^{n})$. By our choice of $(C^s_t, \{\sigma_{i}(t)\}_{i=1}^{n})$, there exist two distinct subcurves of the special fiber $Z_1, Z_2 \subset C^{s}$ and isomorphisms $i_1, i_2:C^{s} \simeq C^s_{t}$ satisfying $i_1(Z_1)=\Exc(\phi_t)$, $i_2(Z_2)=\Exc(\phi_t)$. By Proposition \ref{P:Contractions}, there exist birational contractions
\[
\xymatrix{
&\C^{s} \ar[dr]^{\phi_2} \ar[dl]_{\phi_1}&\\
\C_1\ar[dr]&&\C_2 \ar[dl]\\
&\Delta&\\
}
\]
with $\Exc(\phi_1)=Z_1$, $\Exc(\phi_2)=Z_2$. Since the restriction of $\phi_1$ and $\phi_2$ to the special fiber are contractions, Lemma \ref{L:TwistingX} implies that the special fibers of $\C_1$ and $\C_2$ both lie in $\X$. Since $\X \subset \V_{g,n}$ is open, the maps $\Delta \rightarrow \V_{g,n}$ induced by $\C_1$ and $\C_2$ both factor through $\X$. Since $Z_{1} \neq Z_{2}$, the rational morphism $\C_{1} \dashrightarrow \C_{2}$ does not extend to an isomorphism, and we conclude that $\X$ is not separated, a contradiction.
\end{proof}

\appendix
\section{Stable modular compactifications of $\M_{2}, \M_{3}, \M_{2,1}$}
In this appendix, we give an explicit definition of the relative nef cone $\Nef$ and cone of curves $\EffCurves$ as rational closed convex polyhedral cones in $\Pic_{\Q}(\C/\SM_{g,n})$ and $\Pic_{\Q}(\C/\SM_{g,n})^{\vee}$. In the special cases $(g,n)=(2,0), (3,0), (2,1)$, we enumerate the extremal faces of $\EffCurves$ and describe the corresponding $\Z$-stability conditions, as guaranteed by Lemma \ref{L:NefAssignments}.

To begin, let $\pi:\C \rightarrow \SM_{g,n}$ denote the universal curve over the moduli stack of stable curves over an algebraically closed field of characteristic zero. In this setting, the $\Q$-Picard group of $\SM_{g,n}$ is well-known: we have natural line-bundles $\lambda, \{\psi_i\}, \{ \delta_{i,S} \} \in \Pic(\SM_{g,n})$, where $\lambda=\text{det} (\pi_*\omega_{\C/\SM_{g,n}})$, $\psi_{i}:=\sigma_i^*\omega_{\C/\SM_{g,n}}$, and $\delta_{i,S}$ is the line-bundle corresponding to the reduced irreducible Cartier divisor $\Delta_{i,S} \subset \SM_{g,n}$. Of course, if $i=0$ (resp. $g$), then we must have $|S| \geq 2$ (resp. $|S| \leq n-2$). Since we have a natural identification $\C \simeq \SM_{g,n+1}$, we may define elements $\omega_{\pi}, \{\sigma_i\}, \{E_{i,S}\} \in \Pic(\C)$ by the formulae:
\begin{align*}
\omega_{\pi}:&=\psi_{n+1},\\
\sigma_{i}:&=\delta_{0,i \cup \{n+1\}},&& i \in [1,n] \\\
E_{i,S}:&=\delta_{i,S \cup \{n+1\}}.&& i \in [0,g], S \subset [1,n]
\end{align*}
One should think of $\omega_{\pi}$ as the relative dualizing sheaf of $\pi$, $\sigma_i$ as the line-bundle corresponding to the divisor $\sigma_{i}(\SM_{g,n}) \subset \C$, and $E_{i,S}$ as the line-bundle corresponding to the irreducible component of $\pi^{-1}(\Delta_{i,S})$ whose fibers over $\Delta_{i,S}$ are curves of genus $i$, marked by the points of $S$. Whenever we write $\{\sigma_i\}$, we consider the index $i$ to run between $1$ and $n$, and whenever we write $\{E_{i,S}\}$ we consider $(i,S)$ to run over a set of indices representing each irreducible component of the boundary of $\SM_{g,n}$ once, excluding $\Delta_{irr}$ and $\Delta_{g/2, \emptyset}$.

\begin{lemma}\label{L:RelativePic}
The classes $\omega_{\pi}$, $\{\sigma_i\}$, and $\{E_{i,S}\}$ generate $\Pic_{\Q}(\C/\SM_{g,n})$. Moreover, we have
\begin{itemize}
\item[1.] If $g \geq 2$, these classes freely generate, i.e.
\[
\Pic_{\Q}(\C/\SM_{g,n}) = \Q\{\omega_{\pi}, \{\sigma_i\}, \{E_{i,S}\} \}
\]
\item[2.] If $g=1$, then the classes $\{\sigma_i\}$ and $\{E_{i,S}\}$ freely generate, i.e.
\[
\Pic_{\Q}(\C/\SM_{1,n}) = \Q\{ \{\sigma_i\}, \{E_{i,S}\} \}
\]
\item[3.]If $g=0$, then the classes  $\omega_{\pi}$ and $\{E_{i,S}\}$ freely generate, i.e.
\[
\Pic_{\Q}(\C/\SM_{0,n}) = \Q\{\omega_{\pi}, \{E_{i,S}\} \}.
\]
\end{itemize}
\end{lemma}
\begin{proof}
This follows from the generators and relations for $\Pic(\SM_{g,n}) \otimes \Q$ described in \cite{AC}.
\end{proof}

Now let us recall how these generators intersect irreducible components of fibers of $\pi$ (see, for example, \cite{HarMor}). Let $(C^s,\spn)$ be a fiber of the universal curve $\pi: \C \rightarrow \SM_{g,n}$, and let $G$ be the dual graph of $(C^s,\spn)$. If $Z \subset C^{s}$ is an irreducible component, corresponding to the vertex $v \in G$, then we have

\begin{align*}
\omega_{\pi}.Z&=2g(v)-2+|v|\\
\sigma_{i}.Z&=
\begin{cases}
1&\text{if $v$ is labelled by $p_i$,}\\
0&\text{otherwise.}\\
\end{cases}\\
E_{i,S}.Z&=
\begin{cases}
1&\text{if $v$ has an edge of type-$(i,S)$,}\\
-1&\text{if $v$ has an edge of type-$(i,S)^{c}$,}\\
0&\text{otherwise},\\
\end{cases}
\end{align*}
where we say that \emph{$v$ has an edge of type-$(i,S)$} if $v$ meets an edge corresponding to a node that disconnects the curve into pieces of type $(i,S)$ and $(g-i,S)$, and $v$ lies on the piece of type $(g-i,S)$. Given a $\Q$-line bundle 
$$\L:=a\omega_{\pi}+\sum_{i}b_i\sigma_i+\sum_{i,S}c_iE_{i,S}, \text{ where } a, \{b_i\}, \{c_{i,S}\} \in \Q,$$ $\L$ is nef iff, for every dual graph $G$ and every vertex $v \in G$, we have
$$a(\omega_{\pi}.v)+b_i(\sigma_i.v)+c_{i,S}(E_{i,S}.v) \geq 0,$$
where $\omega_{\pi}.v$, $\sigma_i.v$, $E_{i,S}.v$ are defined by the expressions above. We then define the relative nef cone $\Nef \subset \Pic_{\Q}(\C/\SM_{g,n})$ to be the intersection of this finite collection of half-spaces. The fact that $\omega_{\pi}$ is positive on every stable curve implies that these half-spaces have non-empty intersection, hence determine a piecewise-linear closed convex cone. Of course, the \emph{relative cone of curves} is simply defined to be the dual cone $\EffCurves:=\Nef^{\vee} \subset \Pic_{\Q}(\C/\SM_{g,n})^{\vee}$.

Let us see how this works in practice by computing the relative cone of curves for $\SM_{2}$, $\SM_{3}$, and $\SM_{2,1}$, and describe the stability condition corresponding to each face: already, in these low-genus examples, one sees many new stability conditions that have no counterpart in the existing literature. Throughout the following examples, we will make repeated use of the observation that to determine whether a line-bundle is $\pi$-nef, it is sufficient to intersect it against those fibers of $\pi$ which are maximally-degenerate, i.e. those which correspond to zero strata in $\SM_{g,n}$.

\begin{figure}
\scalebox{.50}{\includegraphics{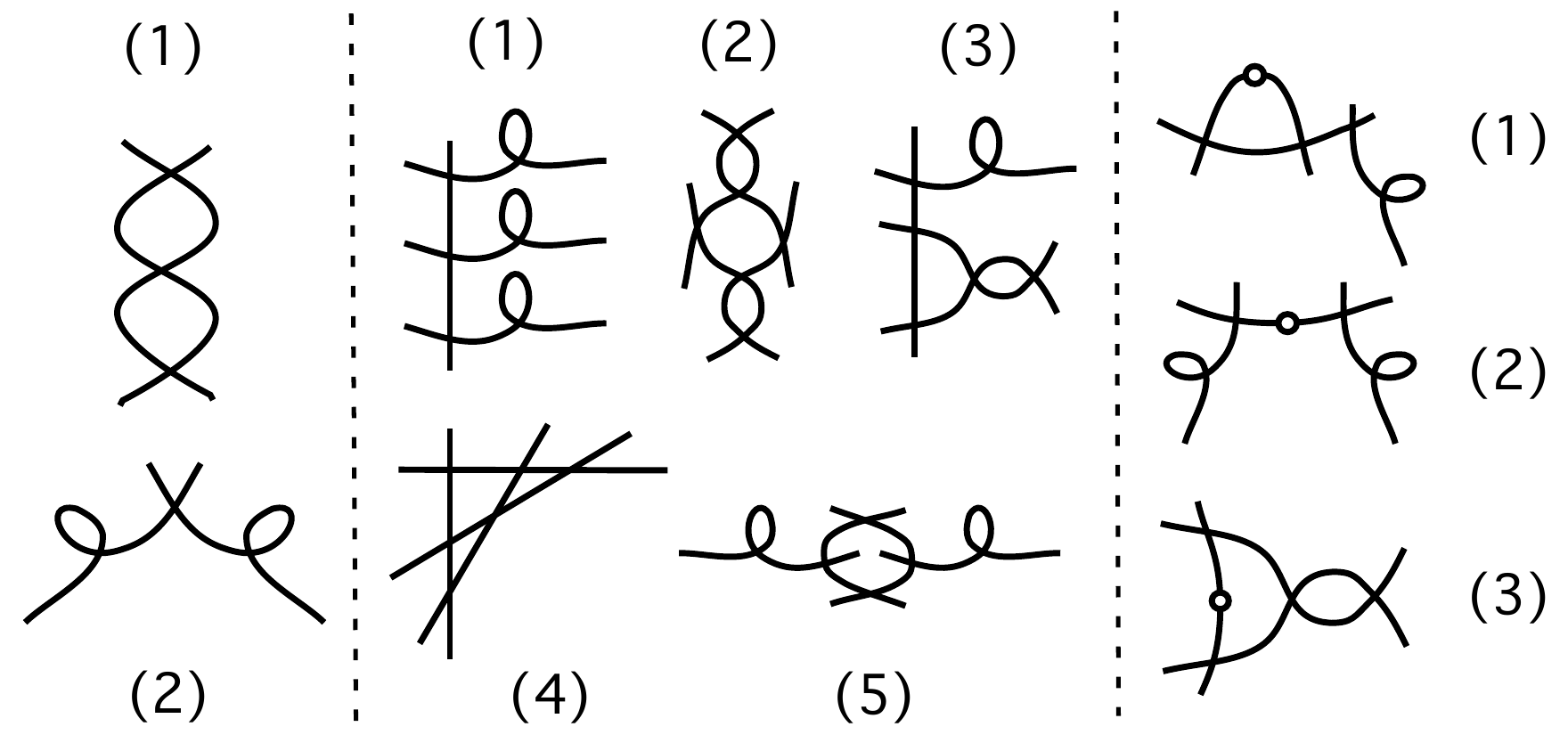}}
\caption{The zero-strata of $\SM_{2}$, $\SM_{3}$, and $\SM_{2,1}$.}\label{F:ZeroStrata}
\end{figure}

\begin{example}[$\M_{2}$]\label{E:M2}
By Lemma \ref{L:RelativePic}, the relative $\Q$-Picard group of the universal curve $\C \rightarrow \SM_{2}$ is given by
$$
\Pic_{\Q}(\C/\SM_{2})=\Q\{\omega_{\pi}\}.
$$
Thus, any numerically non-trivial $\pi$-nef line bundle on $\C$ is ample, and induces the trivial extremal assignment $\Z(G) = \emptyset$.

In fact, it is easy to verify that there are no non-trivial extremal assignments over $\SM_{2}$ directly from the axioms. Let $G_1$,  $G_2$ be dual graphs corresponding to thee two zero-strata pictured in Figure \ref{F:ZeroStrata}. By axiom 2, any extremal assignment which picks out one vertex from $G_1$ (or $G_2$) must pick out both vertices, which contradicts axiom 1. We conclude that $\Z(G_1)=\Z(G_2)=\emptyset$. By axiom 3, $\Z$ must be the trivial extremal assignment. In sum, $\SM_{2}$ is the unique stable modular compactification of $\M_{2}$.
\end{example}

\begin{example}[$\M_{3}$]\label{E:M3} By Lemma \ref{L:RelativePic}, the relative $\Q$-Picard group of the universal curve $\C \rightarrow \SM_{3}$ is given by
$$
\Pic_{\Q}(\C/\SM_{3})=\Q\{\omega_{\pi}, E\},
$$
where $E:=E_{1}$ is the divisor of elliptic tails in the universal curve. Intersecting the divisor $a\omega_{\pi}+bE$ ($a,b \in \Q$) with the irreducible components of vital stratum (1) in Figure \ref{F:ZeroStrata}, we deduce the inequalities $a+3b \geq 0$ and $a-b \geq 0$. One easily checks that any divisor whose coefficients satisfy these two inequalities automatically satisfies the inequalities arising from the irreducible components in strata (2)-(4). Thus, the nef cone $\overline{\text{N}}^{1}_{+}(\C/\SM_{3}) \subset \Pic_{\Q}(\C/\SM_{3})$ is defined by
$$
\overline{\text{N}}^{1}_{+}(\C/\SM_{3})=\Q_{\geq 0}\{\omega_{\pi}(E), \omega_{\pi}(-E/3)\}.
$$
Thus, the relative cone of curves has two extremal faces, namely $\omega_{\pi}(-E/3)^{\perp}$ and $\omega_{\pi}(E)^{\perp}$. One easily checks that the nef divisor $\omega_{\pi}(E)$ has degree zero on an irreducible component of a fiber of the universal curve $\C \rightarrow \SM_{3}$ iff this component is contained in the divisor $E$ (i.e. if it is an elliptic tail). Thus, the extremal assignment induced by this divisor coincides with Example \ref{E:FirstAssignments} (2), and the corresponding moduli space replaces elliptic tails by cusps.

On the other hand, one easily checks that $\omega_{\pi}(-E/3)$ has degree zero on a fiber of $\C \rightarrow \SM_{3}$ iff it has the form $R \cup E_1 \cup E_2 \cup E_3$, where $R$ is a smooth rational curve attached to three distinct elliptic tails $E_1$, $E_2$, and $E_3$. Since the unique singularity of type $(0,3)$ is the rational triple point (i.e. the union of the 3 coordinate axes in $\mathbb{A}^{3}$), such curves are replaced in $\SM_{3}(\Z)$ by curves of the form $E_{1} \cup E_{2} \cup E_{3}$ where the three elliptic tails meet in a rational triple point.
\end{example}

\begin{figure}
\scalebox{.50}{\includegraphics{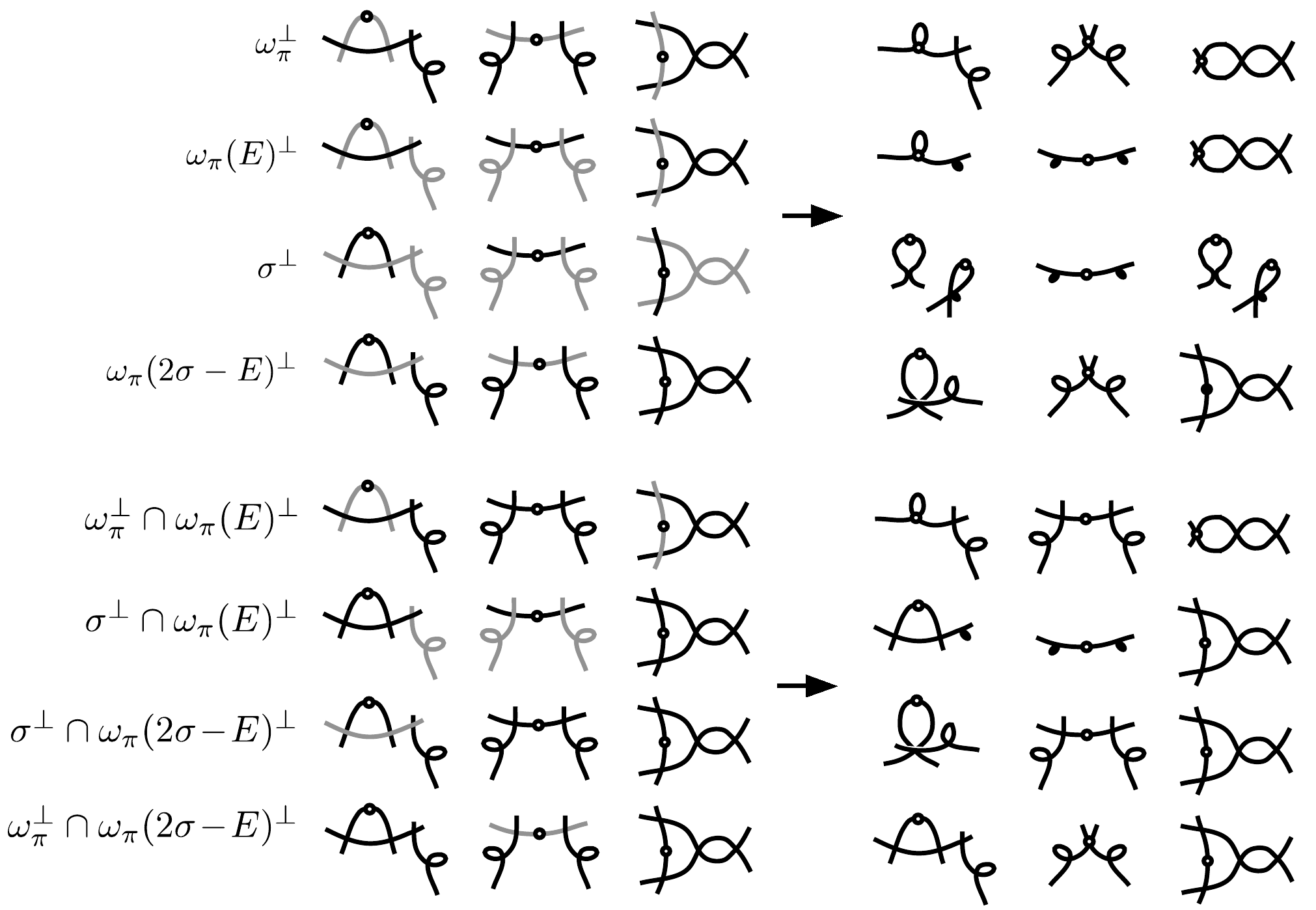}}
\caption{Faces of the relative cone of curves $\overline{\text{N}}_{1}^{+}(\C/\SM_{2,1})$.}\label{F:M_{2,1}Graphs}
\end{figure}
\begin{example}[$\M_{2,1}$]\label{E:M21}
By Lemma \ref{L:RelativePic}, the relative $\Q$-Picard group of the universal curve $\C \rightarrow \SM_{2,1}$ is given by
$$
\Pic_{\Q}(\C/\SM_{2,1})=\Q\{\omega_{\pi}, \sigma, E\},
$$
where $\sigma:=\sigma_1$ is the universal section, and $E:=E_{1, \emptyset}$ is the divisor of unmarked elliptic tails in the universal curve. Intersecting the divisor $a\omega_{\pi}+b\sigma+cE$ ($a,b,c \in \Q$) with the three irreducible components of vital strata (1)-(3) in Figure \ref{F:ZeroStrata}, we deduce the following inequalities for the nef cone:
\begin{eqnarray*}
\text{\underline{Stratum 1}:}&\text{\underline{Stratum 2}:}&\text{\underline{Stratum 3}:}\\
b \geq 0 & a-b \geq 0 & b \geq 0\\
a+c \geq 0& b+2c \geq 0 & a \geq 0\\
a-c \geq 0 & b+2c \geq 0 & a \geq 0
\end{eqnarray*}
One easily checks that this intersection of half-spaces is simply the polyhedral cone generated by the vectors $\{ (1,0,0), (1,0,1), (0,1,0), (1,2,-1) \}$. Thus, the nef cone $\overline{\text{N}}^{1}_{+}(\C/\SM_{2,1}) \subset \Pic_{\Q}(\C/\SM_{2,1})$ is defined by
$$
\overline{\text{N}}^{1}_{+}(\C/\SM_{2,1})=\Q_{\geq 0}\{\omega_{\pi},\, \omega_{\pi}(E),\, \sigma,\, \omega_{\pi}(2\sigma-E)\}.
$$
It follows that the cone of curves $\overline{\text{N}}_{1}^{+}(\C/\SM_{2,1})$ has eight extremal faces: the codimension-one faces are given by $\omega_{\pi}^{\perp},\, \omega_{\pi}(E)^{\perp},\, \sigma^{\perp},\, \omega_{\pi}(2\sigma-E)^{\perp}$, while the codimension-two faces are given by $\omega_{\pi}^{\perp} \cap \omega_{\pi}(E)^{\perp},\, \omega_{\pi}(E)^{\perp} \cap \sigma^{\perp},\,  \sigma^{\perp} \cap \omega_{\pi}(2\sigma-E)^{\perp},\, \omega_{\pi}(2\sigma-E)^{\perp} \cap \omega_{\pi}^{\perp}$. 

The irreducible components of the vital strata (1)-(3) contained in these faces are displayed in Figure \ref{F:M_{2,1}Graphs}. In addition, we have indicated the singular curves that arise in the alternate moduli functors associated to these faces: For example, associated to $\omega_{\pi}^{\perp}$, we see only nodal curves, but the marked point is allowed to pass through the node. Associated to $\omega_{\pi}(E)^{\perp}$, we see the same phenomenon as well as elliptic tails replaced by cusps. Associated to $\sigma^{\perp}$, we see genus-one bridges being replaced by the two isomorphism classes of singularities of type (1,2), namely tacnodes and a planar cusp with a smooth transverse branch. Finally, associated to $\omega_{\pi}(2\sigma-E)^{\perp}$, we see both an unmarked rational curve replaced by a rational triple point and a marked rational curve replaced by a marked node.

\begin{comment}
We should remark that for $\Z=\omega_{\pi}(2\sigma-E)^{\perp}$, the associated rational map $\phi: \SM_{2,1} \dashrightarrow \SM_{2,1}(\Z)$ is not regular. This can be see as follows: Let $C^s=E \cup R$ be the reducible stable curve consisting of an elliptic bridge and a smooth rational bridge, with the marked point contained on $R$. Let $\C^s \rightarrow \Delta$ be a one-parameter smoothing of $C^s$, and let $\C \rightarrow \Delta$ be the associated $\Z$-stable model of the generic fiber, obtained by contracting $E$. If $\phi$ were regular, then the central fiber $C$ would be independent of the smoothing. One can check, however, that there exist smoothings in which contraction of the elliptic bridge produces a tacnode and others where it produces a cusp with transverse branch. Thus, $\phi$ cannot be regular.

A straightforward albeit tedious combinatorial argument, similar to what was carried out in the previous example, shows that in fact that there every non-trivial extremal assignment over $\SM_{2,1}$ is one of the eight listed above.
\end{comment}

\end{example}

\section{The moduli stack of (all) curves (by Jack
  Hall)}\label{S:Stack} 
The purpose of this appendix is to prove that the moduli stack of
curves is algebraic. The statement is well-known to experts and
appears in \cite[Proposition 2.3]{DHS}, but it seems worthwhile to
give a self-contained proof which does not rely on the theory of
Artin approximation. In addition, corollaries \ref{C:Stackgn} and
\ref{C:Stackgne} are used in the main body of the paper and do not
appear elsewhere in the literature. Throughout this appendix, we
follow the notations and terminology of \cite{LMB}. 

Let $(\Aff/S)$ denote the category of affine $S$-schemes. If $S =
\Spec \mathbb{Z}$, we write simply $(\Aff)$. We define a category
$\U$, fibered in groupoids over $(\Aff)$, whose objects are flat,
proper, finitely-presented morphisms of relative dimension one $\C \to
S$, where $\C$ is an algebraic space and $S$ is an affine scheme, and
whose arrows are Cartesian diagrams: 
  \[
  \xymatrix{\C' \ar[r] \ar[d] & \C \ar[d] \\ S' \ar[r] & S}
  \]
Clearly, $\U$ is a stack over $\Spec \mathbb{Z}$. We will prove
\begin{theorem}\label{T:Stack}
  $\U$ is an algebraic stack, locally of
  finite type over $\Spec \mathbb{Z}$.
\end{theorem}
\begin{proof}[Proof of Theorem \ref{T:Stack}]
To prove that $\U$ is algebraic and
  locally of finite type over $\mathbb{Z}$, we must show:
  \begin{enumerate} 
  \item the diagonal $\Delta : \U \to \U\times \U$ is
    representable, quasi-compact, and separated. This is done in
    Section B.1.
  \item There is an algebraic space $U$, locally of finite type over
    $\Spec \mathbb{Z}$, together with a smooth, surjective 1-morphism $U \to
    \U$. This is done in Section B.2.
  \end{enumerate}
\end{proof}
Let $\U_{g,n}$ (resp. $\U_{g,n,d}$) be the stack of flat, proper,
finitely presented morphisms $\C \to S$, together with $n$
sections, whose geometric fibers are reduced, connected curves of
arithmetic genus $g$ (resp. with no more than $d$ irreducible
components). 

The following corollaries of Theorem \ref{T:Stack} will be proved in
Section B.3.
\begin{corollary}\label{C:Stackgn}
  $\U_{g,n}$ is an algebraic stack, locally of finite type over
  $\Spec \mathbb{Z}$.  
\end{corollary}
\begin{corollary}\label{C:Stackgne}
  $\U_{g,n,d}$ is an algebraic stack, of finite type over
  $\Spec\mathbb{Z}$.  
\end{corollary}

Throughout this appendix, we make free use of the fact that $\U$ is
limit-preserving. Since any ring can be written as an inductive limit
of finitely-generated $\mathbb{Z}$-algebras, this allows one to check
properties of $\U$ using test schemes which are of finite type over
$\mathbb{Z}$ (in particular, noetherian). 
\begin{lemma}
If $A = \varinjlim_{i}A_{i}$ is an inductive system of rings, then
there is an equivalence of groupoids:
$$
\U(\Spec A) \simeq \varprojlim_{i} \U(\Spec A_{i}).
$$
\end{lemma}
\begin{proof}
By \cite[4.18.1]{LMB}, the category of finitely-presented algebraic
spaces over affine base-schemes is a limit-preserving stack over
$(\Aff)$. The fact that a proper, flat, relative-dimension one
morphism over $\Spec A$ is induced from a proper, flat,
relative-dimension one morphism over some $\Spec A_{i}$, then follows
from \cite[IV.3.1]{Knutson}, \cite[11.2.6]{EGAIV}, and
\cite[4.1.4]{EGAIV} respectively. 
\end{proof}

\subsubsection*{B.1 Representability of the diagonal}\label{S:Diagonal}
In this section, we prove that the diagonal morphism $\Delta :
\U \to \U \times \U$ is representable, finitely presented and
separated. Equivalently, we must show that if  
$\pi_1 : \C_1 \to S$, $\pi_2 : \C_2 \to S$ are two objects of $\U$,
then the sheaf $\Isom_S(\pi_1,\pi_2)$ is representable by an algebraic
space, finitely presented and separated over $S$.

Recall that if $\pi: \C_1 \rightarrow S$ and $\pi: \C_2 \rightarrow S$
are proper finitely-presented morphisms over an affine scheme $S$ (not
necessarily curves), then $\Hom_S(\pi_1,\pi_2)$ and
$\Isom_S(\pi_1,\pi_2)$ are the sheaves over (\Aff/S) whose sections
over $T \rightarrow S$ are given by $T$-morphisms
(resp. $T$-isomorphisms) $\C_{1} \times_{S} T \to \C_{2} \times_{S}
T$. 

Also recall that if $\C \rightarrow S$ is proper and
finitely-presented, the Hilbert functor $\Hilb_S(\C)$ is the sheaf
over $(\Aff/S)$ whose sections over $T \rightarrow S$ are given by
closed subschemes $\Z \subset \C \times_{S} T$, flat and
finitely-presented over $S$. If $\C \rightarrow S$ is projective, then
$\Hilb_{S}(\C)$ is representable by an $S$-scheme of the form
$\coprod_{p(t) \in \mathbb{Z}[t]}H_{p(t)}$, where $H_{p(t)}$ is a
projective scheme over $S$ parametrizing families with Hilbert
polynomial $p(t)$.

We will show that objects $\C \to S$ of $\U$  are \'etale-locally projective (Lemma
\ref{L:Fppfprojective}) and the representability of
$\Isom_{S}(\pi_1,\pi_2)$ will be deduced from the representability of the Hilbert
scheme for finitely presented projective morphisms. 
\begin{lemma}\label{L:Fppfprojective}
Let $\pi : C \to S$ be a proper, finitely presented morphism of algebraic
  spaces. Let $s\in S$ be a closed point such that $\dim_{\Bbbk(s)}
  C_s \leq 1$, then there is an \'etale neighbourhood 
  $(U,u)$ of $(S,s)$ such that $C\times_S U \to U$ is projective.  
\end{lemma}
\begin{proof}
  The statement is local on $S$ for the \'etale topology and by the
  standard limit methods we reduce immediately to the following situation:
  $S = \Spec R$, where $R$ is an excellent, strictly henselian local
  ring and $s \in S$ is the unique closed point.

  First, we asume that $C$ is a scheme. Take $C_s\to s$ to denote the
  special fiber of $C\to S$. Since $C_s$ is proper and of dimension 1
  over a field, it is manifestly projective. Let $\mathscr{L}_s$ be an
  ample bundle on $C_s$, then by \cite[Proposition 4.1]{SGA4.5}, we
  may lift it to a line bundle $\mathscr{L}$ on $C$. By \cite[Theoreme
  4.7.1]{EGAIII}, one concludes that $\mathscr{L}$ is ample for $C\to
  S$ and thus, $C\to S$ is projective.

  Next, assume that $C$ is a reduced, normal algebraic space, then by
  \cite[Cor. 16.6.2]{LMB}, there is an isomorphism of algebraic
  spaces $[C'/G]\to C$, where $C'$ is a scheme and $G$ is a finite
  group acting freely on $C'$. By the above, $C'$ is a projective and
  since the quotient of a finite group acting freely on a projective
  scheme is projective, one concludes $C$ is projective. 

  Take $C$ to be general, then the usual exponential
  sequence, coupled with basic facts about the  
  \'etale site on $C$, implies that if $C''$ is the normalization
  of the reduction of $C$, the pullback $\Pic C \to \Pic C''$ is
  surjective. Since $S$ is excellent, $C'' \to S$ is proper and
  satisfies the hypotheses of the previous paragraph and is
  consequently projective. Lifting an ample line bundle of $C''$ to
  $C$, we conclude that this lift is ample ($C'' \to C$ is finite) and
  hence $C\to S$ is projective.
\end{proof}
\begin{corollary}\label{C:RepDiagonalStep}\label{C:RepDiagonal}
Suppose that $\pi_1 : \C_1 \to S$, $\pi_2 : \C_2 \to S$ are objects of
$\U$, then the sheaves $\Hom_S(\pi_1,\pi_2)$ and
$\Isom_{S}(\pi_1,\pi_2)$ are both representable by finitely presented
and separated algebraic $S$-spaces.
\end{corollary}
\begin{proof}
By Lemma \ref{L:Fppfprojective}, there is an \'etale surjection
$T \to S$ such that for $i=1$, $2$, the pullbacks, $\pi_{i,T} : \C_{i,T}\to T$, are
projective, flat and finitely presented. 
The inclusions $\Isom_{T}(\pi_{1,T},\pi_{2,T}) \subset
\Hom_{T}(\pi_{1,T},\pi_{2,T}) \subset \Hilb_{T}(\C_{1,T} \times_{T} \C_{2,T})$ are
representable by finitely-presented open immersions.\footnote{$\U$ is
  limit preserving, so we may assume that $T$ is noetherian. The
  assertion for the second inclusion follows from the first. Indeed,
  this inclusion is given by the graph homomorphism and those families
  of closed subschemes of $\C_{1,T} \times_T \C_{2,T}$ which are
  graphs are precisely those families for which projection onto the
  first factor is an isomorphism. The first inclusion is covered by
  \cite[Prop. 4.6.7(ii)]{EGAIII}.} From the existence of the Hilbert
scheme for finitely presented projective morphisms, we make the
following two observations: 
\begin{enumerate}
\item $\Hom_S(\pi_1,\pi_2) \times_{S} T \simeq
  \Hom_T(\pi_{1,T},\pi_{2,T})$ is representable by a separated and
  locally of finite type $S$-scheme. In particular, the morphism
  $\Hom_T(\pi_{1,T},\pi_{2,T}) \to \Hom_S(\pi_1,\pi_2)$ is \'etale and
  surjective. 
\item  The map 
$\Hom_T \times_{\Hom_S }  \Hom_T \rightarrow \Hom_T \times_{S} \Hom_T$
is a quasicompact, closed immersion. Indeed, this is simply the locus
where two separated morphisms of schemes agree.
\end{enumerate}
Putting these together, one concludes (by the definition of an
algebraic space) that $\Hom_S(\pi_1,\pi_2)$ is representable by a
separated and locally of finite type algebraic $S$-space. Since
finitely presented open immersions are local for the \'etale topology,
we deduce the corresponding result for $\Isom_S(\pi_1,\pi_2)$.

All that remains to check is that $\Hom_T$ is quasicompact. Taking a
Segre embedding of the product $\C_{1,T}\times_T \C_{2,T}$ into some
fixed projective space $\mathbb{P}^M_T$, an
application of \cite[Theorem B.7]{FGA_exp} allows one to conclude that
 graphs of morphisms $\C_1 \times_S T \rightarrow \C_2 \times_S T$ all
 have the same Hilbert polynomial and consequently, $\Hom_T$ is
 quasi-projective and in particular, quasi-compact.
\end{proof}
\subsubsection*{B.2 Existence of a smooth cover}\label{S:Cover}
In this section, we prove that $\U$ admits a smooth surjective cover
by a scheme $U \rightarrow \U$. The key point is that if $C$ is a
proper one-dimensional scheme over an algebraically closed field, then
$C$ admits an embedding $C \hookrightarrow \P^n$ such that the natural
map from embedded deformations to abstract deformations is smooth and
surjective (Lemma \ref{L:DefSurjective}) at $[C]$. Thus, we may build
our atlas as a disjoint union of open subschemes of Hilbert
schemes. The key lemma is well-known in the local complete
intersection case, but we need the statement for an arbitrary
one-dimensional scheme, and this requires the cotangent
complex. Throughout this section, $k$ denotes an algebraically closed
ground field. If $X \subset Y$ are $k$-schemes, we let $\Def_X$ and
$\Def_{X \subset Y}$ denote the functor of abstract and embedded
deformations respectively. (For the basic definitions of deformation
theory, the reader may consult \cite{Sernesi}). 

\begin{lemma}\label{L:DefSurjective} Suppose that we have an embedding $i:X \hookrightarrow Y$, where $X$ is a proper $k$-scheme and $Y$ is a smooth $k$-scheme. If $H^1(X,i^*T_{Y/k})=0$, then \emph{$\Def_{X \subset Y} \rightarrow \Def_{X}$} is formally smooth.
\end{lemma}
\begin{proof}
We must show that for any small extension $A' \to A$ with square-zero
  kernel, the map 
  \[
  \Def_{X \subset Y}(A') \to \Def_{X}(A') \times_{\Def_X(A)} \Def_{X
    \subset Y}(A)
  \]
  is surjective.
  Consider a diagram
  \[
  \xymatrix{
  X_{A} \ar@{^{(}->}[r]  \ar[d]  & Y_{A}:=Y \times_{\Spec k} \Spec A \ar[d] \\
  X_{A'} \ar@{-->}[r] & Y_{A'}:=Y \times_{\Spec k}  \Spec A'\\
  }
  \]
  where $[X_{A'}]  \in \Def_{X}(A')$ and $[X_{A} \hookrightarrow Y_{A}] \in \Def_{X \subset Y}(A)$ each restrict to $[X_{A}] \in \Def_{X}(A)$. We must show that there exists a map $X_{A'} \rightarrow Y_{A'}$ (any such map is automatically a closed embedding).

Let us abuse notation by writing $A$ (resp. \!$A'$) for $\Spec A$ (resp. \!$\Spec A'$). The triple
$X_A \xrightarrow{f} Y_{A'} \to A'$ gives a distinguished triangle of complexes \cite[II.2.1.2]{Illusie}:
\[
\xymatrix{f^*L_{Y_{A'}/A'} \ar[r] &  L_{X_A/A'} \ar[r] &  L_{X_A/Y_{A'}}},
\]
and the long exact sequence associated to
$\hom_{X_A}(-,\O_{X_A})$ gives:
\[
\xymatrix{\Ext^1_{X_A}(L_{X_A/Y_{A'}}, \O_{X_A})\ar[r]  & \ar[r]
  \Ext^1_{X_A}(L_{X_A/A'},\O_{X_A}) &
  \Ext^1_{X_A}(f^*L_{Y_{A'}/A'},\O_{X_A})}.  
\]
We claim that  $\Ext^1_{X_A}(f^*L_{Y_{A'}/A'},\O_{X_A})=0$. Since $Y_{A'} \to A'$ is smooth, $L_{Y_{A'}/A'} \cong
\Omega_{Y_{A'}/A'}$ and
$
\Ext^1_{X_A}(f^*L_{Y_{A'}/A'},\O_{X_A}) \cong
H^1(X_A,T_{Y_{A}/A}\mid_{X_A})
$ \cite[III.3]{Illusie}. Since $T_{Y_{A}/A}$ is flat over $\Spec A$ and $H^1(X,\imath^*T_{Y/k}) = 0$, $H^1(X_A,T_{Y_{A}/A}\mid_{X_A})=0$ by \cite[Exercise III.11.8]{Hartshorne}.

By \cite[III.1.2]{Illusie}, $\Ext^1_{X_A}(L_{X_A/Y_{A'}}, \O_{X_A}) \simeq \Exal_{\O_{Y_{A'}}}(\O_{X_A},\O_{X_A})$ and $\Ext^1_{X_A}(L_{X_A/A'},\O_{X_A}) \simeq \Exal_{\O_{A'}}(\O_{X_A},\O_{X_A})$, so  we obtain a surjection:
\[
\xymatrix{\Exal_{\O_{Y_{A'}}}(\O_{X_A},\O_{X_A}) \ar[r] &
  \Exal_{\O_{A'}}(\O_{X_A},\O_{X_A}) \ar[r] &
0 }.  
\]
Thus, there is an
$\O_{Y_{A'}}$-extension of $\O_{X_A}$ by $\O_{X_A}$ mapping to $[X_{A'}] \in \Exal_{\O_{A'}}(\O_{X_A},\O_{X_A}).$ In particular, there is a map of sheaves $\O_{Y_{A'}}
\to \O_{X_{A'}}$ and consequently a morphism $X' \to Y_{A'}$
extending $X_{A} \to Y_{A'}$. 
\end{proof}

\begin{corollary}
Suppose that $X$ is a projective $k$-scheme with $h^2(X,\O_{X})=0$. If $X \hookrightarrow \P^n$ is any embedding such that $h^1(X,\O_{X}(1))=0$, then \emph{$\Def_{X \subset \P^{n}} \rightarrow \Def_{X}$} is formally smooth.
\end{corollary}
\begin{proof}
By Lemma \ref{L:DefSurjective}, it is sufficient to show that $h^1(C,i^*T_{\P^n})=0$. We have an exact sequence
$$
0 \rightarrow \O_{X} \rightarrow \O_{X}(1)^{\oplus(n+1)} \rightarrow i^*T_{\P^{n}} \rightarrow 0,
$$
and taking cohomology gives
$$
H^1(X,  \O_{X}(1))^{\oplus(n+1)} \rightarrow H^1(X,  i^*T_{\P^{n}}) \rightarrow H^2(X, \O_{X}).
$$
Since $h^1(X,\O_{X}(1))=h^2(X,\O_{X})=0$, we conclude $h^1(X, i^*T_{\P^n})=0$ as desired.
\end{proof}

\begin{corollary}
There exists a scheme $U$, locally of finite-type over $\Spec \mathbb{Z}$, and a smooth surjective cover $U \rightarrow \U$.
\end{corollary}
\begin{proof}
Given a geometric point $x:=[C] \in \U$, consider any embedding $C \hookrightarrow \P^{n}$ such that $h^1(C, i^*T_{\P^n})=0$. By cohomology and base-change \cite[Theorem 12.11]{Hartshorne}, there exists a Zariski-open affine neighborhood
$$
[C] \in U_{x} \subset Hilb(\P^{n})
$$
such that the restriction of the universal embedding $\C \hookrightarrow \P^{n} \times Hilb(\P^{n})$ to $U_{x}$ satisfies $h^1(C_y, i^*T_{\P^n_y})=0$ for every $y \in U_{[C]}$. By Lemma \ref{L:DefSurjective}, the induced map $U_x \rightarrow \U$ is formally smooth. Also, since the diagonal $\U \rightarrow \U \times \U$ is representable and of finite presentation (Corollary \ref{C:RepDiagonal}), the map $U_{x} \rightarrow \U$ is representable and finitely-presented. Since $\U$ is limit-preserving, it suffices to check the smoothness of $U_{x} \rightarrow \U$ on noetherian test schemes, and we conclude that $U_{x} \rightarrow \U$ is smooth by \cite[Theorem 17.14.2]{EGAIV}. 

Now let $Hilb_{g,d,n}$ be the Hilbert scheme (over $\Spec \mathbb{Z}$) of curves of genus $g$ and degree $d$ in $\P^{n}$, and let $U \subset \coprod_{g,d,n}H_{g,d,n}$ be the open subscheme parametrizing curves $[C \subset \P^{n}]$ with $h^1(C, i^*T_{\P^n})=0$. Then $U$ is locally of finite type over $\Spec \mathbb{Z}$, and the arguments of the preceding paragraph show that
$$
U \rightarrow \U
$$
is smooth and surjective.
\end{proof}

\subsubsection*{B.3. Proofs of corollaries}
\begin{proof}[Proof of Corollary \ref{C:Stackgn}]
If $\C \rightarrow S$ is an object of $\U$, the set
$$
\{s \in S \,|\, C_{\overline{s}}\mbox{ is reduced, connected, of arithmetic genus $g$}\}
$$
is open in $S$ by \cite[12.2.1-12.2.3]{EGAIV}. It follows that $\U_{g}$, the stack whose objects are flat proper finitely-presented morphisms of relative dimension one with geometrically connected reduced fibers of arithmetic genus $g$, is an open substack of $\U$.

Now observe that the morphism $\U_{g,1} \rightarrow \U_{g}$ obtained by forgetting a section is representable, finitely-presented, and proper. Indeed, if $S \rightarrow \U_{g}$ is any morphism from an affine scheme, corresponding to a family $\C \rightarrow S$, then the fiber product $S \times_{\U_{g}} \U_{g,1}$ is naturally isomorphic to $\Hom_{S}(S,\C)$, which is represented by $\C$ itself. It follows that $\U_{g,1}$ is an algebraic stack, locally of finite type over $\Spec \mathbb{Z}$. Since
  \[
  \U_{g,n} \simeq \underbrace{\U_{g,1} \times_{\U_{g}} \cdots
    \times_{\U_{g}} \U_{g,1}}_n,
  \] 
  we conclude that $\U_{g,n}$ is an algebraic stack, locally of finite-type over $\Spec \mathbb{Z}$.
\end{proof}

For the second corollary, we need the following boundedness lemma
\begin{lemma}(with Smyth, Vakil, and van der Wyck)\label{L:Boundedness}
There exists integers $N_{g,e}$ and $D_{g,e}$ depending only on $g$ and $e$ such that any reduced curve of arithmetic genus $g$ with no more than $e$ irreducible components admits a degree $d$ embedding into $\P^{n}$ with $d \leq D_{g,e}$ and $n \leq N_{g,e}.$
\end{lemma}
\begin{proof}
First, we show there exists an integer $D_{g,e}$ such that any reduced curve of arithmetic genus $g$ with no more than $e$ irreducible components admits a degree $d \leq D_{g,e}$ embedding into some projective space. Given such a curve $C$ with normalization $\pi: \tilde{C} \rightarrow C$, let $Z \subset C$ be an effective Cartier divisor whose support meets the smooth locus of every irreducible component of $C$. Since $C$ has no more than $e$ irreducible components, we may assume that $\deg Z \leq e.$ Let $\L:=\O(Z)$. It suffices to exhibit an integer $m:=m(g,e)$, depending only on $g$ and $e$, such that $\L^{m}$ is very ample on $C$. Indeed, we may take $D_{g,e}=me$. 

To show that $\L^{m}$ separates points and tangent vectors, it is sufficient to show that, for any $p \in C$,
$$
H^1(C,\L^{m} \otimes m_{p})=H^1(C,\L^{m} \otimes m_{p}^2)=0.
$$ 
Clearly, the latter vanishing implies the former. Given $p \in C$, let $\pi^{-1}(p)=p_1+\ldots+p_r$ and let $\delta(p)$ denote the $\delta$-invariant of $p$. We have an exact sequence
$$
\xymatrix{0 \ar[r] &  \pi_*\O_{\tilde{C}}(-2\delta(p)(p_1+\ldots+p_r))
  \ar[r] &  m_{p}^2 \ar[r] &  \mathscr{E}\ar[r] & 0},
$$
where $\mathscr{E}$ is a coherent sheaf supported at $p$. Twisting by $\L^{m}$ and taking cohomology, we obtain
$$
\xymatrix{H^1(C, \L^{m} \otimes
  \pi_*\O_{\tilde{C}}(-2\delta(p)(p_1+\ldots+p_r))) \ar[r] &  H^1(C,
  \L^{m} \otimes m_{p}^2 ) \ar[r] &  0}.  
$$
By the projection formula, we have 
$$H^1(C, \L^{m} \otimes  \pi_*\O_{\tilde{C}}(-2\delta(p)(p_1+\ldots+p_r)))=H^1(\tilde{C}, (\pi^*\L)^m(-2\delta(p)(p_1+\ldots+p_r))),$$
which vanishes as soon as $m>2g-2+2\delta(p)r$. Since $\delta(p) \leq g+e-1$ and $r \leq \delta(p)+1$, we may take $m(g,e):=2g-2+2(g+e)(g+e-1)$.

Next, we must bound the ambient dimension of the projective space. Given a degree $d$ embedding $C \hookrightarrow \P^{n}$, let $\Sec(C)$ and $\Tan(C)$ denote the secant and tangent varieties of $C$ respectively. As long as $n>\max\{\Sec(C), \Tan(C)\}$, we may obtain a lower-dimensional embedding by projection. The dimension of $\Sec(C)$ is at most 3, and the dimension of $\Tan(C)$ is bounded by the maximum embedding dimension of any singular point on $C$. Thus, it suffices to exhibit an integer $N:=N(g,e)$ such that any singularity of $p \in C$ has embedding dimension at most $N$.

Suppose first that $p \in C$ is a unibranch singularity, so $\O_{C,p}$ is a finitely-generated subalgebra of the power series ring $k[[t]]$. Choose a set of generators for $\O_{C,p}$, say $f_1, \ldots, f_m$, and let $N = \min_{i} \{ \deg f_i \}$. Clearly we may choose our generators such that the residues of $\{ \deg f_i \}$ modulo $N$ are all distinct. In particular, we have $m \leq N.$ Since $N \leq \delta(p) \leq g+e-1$, we have the bound $N(g,e)=g+e-1$ for the unibranch case.

If $p \in C$ has $r$ branches, then we have $\O_{C,p} \subset k[[t_1]]\oplus \ldots \oplus k[[t_r]]$. The restriction of $\O_{C,p}$ to each branch is generated by at most $g+e-1$ elements, and there are at most $\delta(p)+1 \leq g+e$ branches, so altogether $\O_{C,p}$ must be generated by at most $(g+e)(g+e-1)$ elements. Thus, we may take $N(g,e)=(g+e)(g+e-1)$.

\end{proof}

\begin{proof}[Proof of Corollary \ref{C:Stackgne}]
The stack $\U_{g,n,d}$ is an open substack of $\U_{g,n}$ by \cite{EGAIV}[12.2.1 (xi)]. It only remains to see that we can cover $\U_{g,n,d}$ by a scheme of finite-type over $\Spec \mathbb{Z}$. Since the Hilbert schemes of curves in $\P^{n}$ of arithmetic genus $g$ and degree $d$ is of finite-type, this follows immediately from Lemma \ref{L:Boundedness}.
\end{proof}

\end{document}